\documentclass[10pt]{amsart}
\usepackage[utf8]{inputenc}
\usepackage{graphicx}
\usepackage{amsmath,amssymb, mathtools}
\usepackage{mathrsfs, esint, bbm}
\usepackage{multicol}
\usepackage{hyperref}
\usepackage{esint}
\usepackage{float}
\restylefloat{table}
\usepackage[margin=3cm]{geometry}

\usepackage[normalem]{ulem} 

\usepackage[nameinlink,capitalise,sort]{cleveref}
\crefname{equation}{}{} 
\crefname{enumi}{}{} 

\newtheorem{example}{Example}[section]

\newtheorem{theorem}{Theorem}[section]
\newtheorem{lemma}{Lemma}[section]
\newtheorem{proposition}{Proposition}[section]

\newtheorem*{maintheorem*}{Main Theorem}
\allowdisplaybreaks
\numberwithin{equation}{section}

\usepackage{todonotes}

\newcommand{\N}{\mathbb{N}}

\newcommand{\R}{\mathbb{R}}

\newcommand{\abs}[1]{\left|#1\right|}
\newcommand{\norm}[1]{\left\|#1\right\|}

\newcommand{\eps}{\varepsilon}

\DeclareMathOperator*{\supp}{supp}

\newcommand{\dd}{\,\mathrm{d}}
\renewcommand{\d}{\mathrm{d}}
\newcommand{\Ds}{\left(-\Delta\right)^s }

\newcommand{\Dusm}{\left(-\Delta\right)^{\frac{1+s}{2}} }
\makeindex

\renewcommand{\rm}{\mathrm}

\renewcommand{\div}{\operatorname{div}}

\begin{document}

\title[Existence of solutions for a fractional thin-film equation]{Existence and finite speed of propagation of solutions for a
multidimensional fractional thin-film equation}


\author[N.~De Nitti]{Nicola De Nitti}
\address[N.~De Nitti]{Università di Pisa, Dipartimento di Matematica, Largo Bruno Pontecorvo 5, 56127 Pisa, Italy.}
\email[]{nicola.denitti@unipi.it}

\author[S.~Lisini]{Stefano Lisini}
\address[S.~Lisini]{Università degli Studi di Pavia, Dipartimento di Matematica Felice Casorati, Via Ferrata 5, 27100 Pavia, Italy.}
\email[]{stefano.lisini@unipv.it}

\author[A.~Segatti]{Antonio Segatti}
\address[A.~Segatti]{Università degli Studi di Pavia, Dipartimento di Matematica Felice Casorati, Via Ferrata 5, 27100 Pavia, Italy.}
\email[]{antonio.segatti@unipv.it}

\author[R.~Taranets]{Roman Taranets}
\address[R.~Taranets]{Institute of Applied Mathematics and Mechanics of the NAS of Ukraine,
G.~Batiuka Str.~19, 84116 Sloviansk, Ukraine.}
\email[]{taranets\_r@yahoo.com}

\begin{abstract}
In this paper, we discuss existence and finite speed of propagation for the solutions to an initial-boundary value problem for
a family of fractional thin-film equations
in a bounded domain in $\R^d$. The nonlocal operator we consider is the spectral fractional Laplacian with Neumann boundary conditions. In the case of a ``strong slippage'' regime with ``complete wetting'' interfacial conditions, we prove local entropy estimates that entail finite speed of propagation of the support and a lower bound for the waiting time phenomenon.
\end{abstract}

\maketitle

\section{Introduction}
\label{sec:intro}

We study the existence and asymptotic behavior of 
solutions to the initial boundary value problem (IBVP)
\begin{equation}\label{eq:ft}
	\begin{cases}
	\partial_t u(x,t)= \div(u^n(x,t)\nabla p(x,t)), & (x,t) \in  \Omega_T ,\\
	p(x,t)= \Ds u(x,t), &  (x,t) \in \Omega_T,\\
	\nabla u \cdot \textbf{n} =   \nabla p \cdot \textbf{n} = 0 , &  (x,t) \in \partial \Omega \times (0,T),\\
	u(x,0) =u_0(x), & x \in \Omega,
\end{cases}
\end{equation}
where  $\Omega \subset \R^d $ is a bounded domain with smooth boundary $\partial \Omega$,  $\textbf{n}$ is the exterior normal vector to
$\partial\Omega$, $\Omega_T \coloneqq  \Omega \times (0,T)$,
 $d \in \N$, $n > 0$, and $s\in(0,1)$. The pressure $p$ is related to the unknown function $u$ via the \emph{spectral fractional Laplacian} (with Neumann boundary conditions), which is defined as
\begin{equation}\label{eq:fl}
(-\Delta)^s u(x) \coloneqq   \mathop {\sum}_{k =0}^{+\infty} {  \lambda_k^{s} (u,\phi_k) \phi_k(x)},
\end{equation}
where $\lambda_k$ and $\phi_k$ are the eigenvalues and the normalized eigenfunctions of the Laplacian on $\Omega$ with the homogeneous Neumann boundary conditions, $(u,v)$ denotes the scalar product in $L^2(\Omega)$. We refer to \cite{StingaTorrea,AbatangeloValdinoci} for further information on this operator.

The PDE in \cref{eq:ft} is a nonlocal degenerate parabolic-type equation of order $2(s + 1)$ that arises in modeling hydraulic fractures: the parameter $s \in (0,1)$ is related to the properties of the medium in which the fracture spreads, while the power
$n > 0$ corresponds to different slip conditions on the fluid–solid interface.

It is worthwhile noting that \cref{eq:ft} interpolates between two well-known equations, namely the {\itshape porous medium equation} (PME) (when, formally, $s=0$) and the
{\itshape thin film equation} (TFE) (when, formally, $s=1$).

We refer to the books \cite{VazquezBook,VazquezBook2} for the wide literature regarding the porous medium equations. For the TFE ($s=1$), the case $n \in (1,2)$ corresponds to ``strong slippage'';  $n \in (2,3)$ to ``weak slippage'';  $n = 3$  to a ``no-slip condition'' \cite{RevModPhys.69.931};  $n =2$ to the ``Navier-slip condition'' \cite{JAGER200196}; the case  $n = 1$  arises as the lubrication approximation of the Hele--Shaw flow \cite{GO}.

The properties of solutions to this class of equations strongly depend on the value of the parameter $n >0$. The mathematical study of the simplest one-dimensional model  \cref{eq:ft}, with $s=1$, was initiated by Bernis and Friedman in \cite{B8}. They derived a
positivity property and proved the existence of non-negative weak generalized
solutions of initial-boundary problem for non-negative initial data.
More regular (strong or entropy) solutions have been constructed, for $d = 1$, in
\cite{BerettaBertschDalPasso,BP1996}; for $d\in\{2,3\}$, in \cite{DPGG98SIAM}; and, for $d \geqslant  4$, in \cite{Li9}.

Many interesting qualitative properties of the solutions have been discovered and investigated. The finite speed
of propagation (FSP) of the support of the solution was established,  for  $d =1$ and $0< n < 2$ in \cite{BernisFinite} or $2 \leqslant   n < 3$ in \cite{BernisFinite2}, or $n \geqslant  4 $ in \cite{B8}; for $d\in\{2,3\}$ and $\frac{1}{8}< n < 2$ in \cite{ThinViscous} or $2 \leqslant   n < 3$ in \cite{Grun2003}. The finite \emph{backward} speed of propagation property for $d = 1$ and $\frac{1}{2} < n < \frac{3}{2}$
\cite{BernisFinite}.

Another subtle issue is the  \emph{waiting-time phenomenon} (WTP): proving that, if the initial data  are ``flat enough'' near some point $x_0$ on the initial free boundary, the free boundary locally remains stationary (or at most move backward) for some time
before it finally starts moving forward.

Sufficient conditions for the WTP were obtained, for $d=1$ and $0 < n < 3$ and for $d \in\{2,3\}$ and $\frac{1}{8} < n < 2$ in \cite{DalPassoGiacomelliGruen} or for $ 2 \leqslant   n < 3$ in \cite{GruenWTWS}. On the other hand, necessary conditions for instantaneous forward motion of the free boundary were first deduced in \cite{FischerAHP,FischerARMA,FischerJDE}. 
Finally, sharp criteria in terms of the mass of the initial data were proven in \cite{MR4444309}.

The uniqueness of the solutions to \cref{eq:ft} with $s =1$  is an open and challenging problem.
Only partial results are available: see \cite{John,MajdoubMasmoudiTayachi,MR2944624}.

The case $d=1$ and $s = \frac{1}{2}$ is considered in \cite[Section 2]{IM11} as a model describing hydraulic fractures driven by a viscous fluid in a uniform elastic
medium under the condition of plane strain. A physical discussion about an impermeable KGD (Khristianovic--Geertsma--de Klerk) fracture model is also contained in \cite{khristianovic1955,geertsma1969}.

 The IBVP \cref{eq:ft} with $s=\frac{1}{2}$ and $n=3$ was derived in \cite{SpenceSharp85} to describe the cracking in an elastic medium with ``no-slip regime''; the model was later generalized for the other slip regimes, i.\,e. $n \geqslant  1$, in \cite{IM11}.
Moreover, as shown in \cite{IM11,MR3397309}, the value $n = 4$ is critical for \cref{eq:ft} (while
the value $n = 3$ is known to be critical for the TFE). For $s = \frac{1}{2}$ and $n\in [1,4)$, in \cite{MR3406645}, Imbert and Mellet were also able to construct self-similar solutions.
In \cite{IM11}, the authors proved the existence of a non-negative weak solution to \cref{eq:ft} for $d =1$,
$s = \frac{1}{2}$ and $n \geqslant  1$. This result was generalized for $d =1$, $s \in (0,1)$
and $n \geqslant  1$ in \cite{MR3397309}.

In contrast to the one-dimensional situation, the multi-dimensional case is less understood.

For the particular case $n=1$ and $d \geqslant  1$, in \cite{Se_Va}  the authors proved the existence of non-negative solutions of the  Cauchy problem for \cref{eq:ft},
where $(-\Delta)^s$ denotes the (singular integral) fractional Laplacian of order $s \in (0,1)$ in $\R^d$ (see \cite{MR3613319,AbatangeloValdinoci}).

They also constructed explicit compactly supported and non-negative self-similar solutions that match the classical Barenblatt profile for the porous medium equation when $s=0$ (see \cite{MR0046217}) and  the Barenblatt-type profile
obtained for zero contact-angle solutions of the TFE in \cite{MR1148286} (for $d=1$) and in \cite{MR1479525} (for $d \geqslant  1$).

These self-similar solutions are then shown to be related to the long-time dynamics under some extra integrability assumption.
The self-similar solutions are constructed among the ones with connected positivity set. Moreover, they are unique in this class.
The uniqueness of the self-similar solutions {\itshape without} assuming that their positivity set is connected has been recently discussed in
\cite{De-Va23}. This uniqueness result is indeed a consequence of the fact, proved in \cite{De-Va23}, that some nonlocal seminorms are strictly decreasing under the continuous Steiner rearrangement.
For our equation (we recall, equation \cref{eq:ft} in $\R^d$ with $n=1$) the nonlocal seminar involved is the $H^s(\R^d)$-seminorm, $s\in (0,1)$.
The reason why the $H^s$-seminorm has such a prominent role is that equation \cref{eq:ft} in $\R^d$ with $n=1$ is
indeed the Wasserstein gradient flow of the $H^s$-seminorm, as recently proved in \cite{Lisini24}.
For \cref{eq:ft}, with $n \geqslant 1$ and $d=1$, self-similar solutions were constructed in \cite{MR3406645}.

In general, proving whether compactly
 supported initial data generate compactly supported solutions and studying the WTP are difficult questions when dealing with \cref{eq:ft}: the main difficulty lies in the fact that the equation is formally
of order $2(s+1)$ and thus comparison arguments are not available. The first results on these questions were obtained in \cite{NT2024}, where  the authors proved  FSP and WTP for weak solutions to \cref{eq:ft} with $d=1$, $s \in (0,1)$, and $n \in (1,2)$.

The aim of this paper is twofold. First, we generalize the existence result \cite{IM11,MR3397309} dealing
with the case of a bounded domain in $\R^d$. Then, we prove, in the multi-dimensional case, the FSP and the WTP, thus extending the contribution of \cite{NT2024}.

\subsection{Outline of the paper}

Let us outline the results of this paper.
In \cref{PSP}, we review some preliminary results on the fractional
Laplacian operator and fractional Sobolev spaces.

In \cref{sec:main}, we formulate our main results on the existence of weak solutions, finite speed of propagation  and waiting time phenomenon.

In the existence result, \cref{Th-ex}, we require that $ n\in \left[1,\frac{d+2(1-s)}{(d-2s)_{+}}\right)$ if $s\in (0,1)$.
We stress that the upper bound $\frac{d+2(1-s)}{(d-2s)_{+}}$ on the exponent $n$ does not coincide with the optimal one for the TFE as $s \nearrow 1$ (i.\,e., it is not optimal). This is associated with the admissible regularity of the flux $u^n \nabla p(u)$ and with the lack, for the moment, of the so-called $\alpha$-entropy estimates. We expect that these estimates could be crucial to obtain the optimal result.

In the results on the interface evolution properties, \cref{th:speed,th:wt}, the restriction $n < \frac{s+2}{s+1}$ for $d \in \{2,3\}$ arises in the proof of a local entropy estimate (\cref{lemma-lee}), as the weak solution is not uniformly bounded from above.
Again, we hope that this problem can be eliminated with the proper $\alpha$-entropy estimates, as in the case of the TFE (see \cite{DPGG98SIAM}).

In \cref{th:speed}, we obtain an interface propagation rate, $d(t) \lesssim t^{\frac{1}{n d + 2(s+1)}}$, which matches that of the self-similar solutions in \cite{MR3406645,Se_Va}.

The sufficient condition on the initial data
for the WTP deduced in \cref{th:wt}, which require a growth up to $| x - x_0|^{\frac{2(s+1)}{n}} $,  appears sharp up to scaling. For example, for $d=1$, there exists the explicit waiting time solution $(1 + c\,t)^{-\frac{1}{n}} x^{\frac{2(s+1)}{n}}$ of \cref{eq:ft}, where $c > 0$ is some constant depending on $n$ and $s$ (cf. \cite[Remark 3.2]{NT2024}).

\cref{sec:existence} is devoted the proof of  \cref{Th-ex}.
In contrast to \cite{IM11,MR3397309,Se_Va,Lisini24}, we use the approach developed
in the theory of multi-dimensional thin film equations (see, e.\,g.,
\cite{Grun95,DPGG98SIAM,DalPassoGiacomelliShishkov,TK14,TS04} etc.). 
In particular, we introduce a nested approximation scheme based on the a proper regularization of the mobility $u^n$ close to zero and at infinity. 
The approximated problem is then (locally) solved using the Faedo-Galerkin scheme. 
Thanks to entropy and energy estimates and compactness, we construct a global in time weak solution.

In \cref{LEE} we obtain a local entropy estimate (\cref{lemma-lee}). This is the main tool in the proof of FSP and WTP.
In \cite{NT2024}, the authors proved this estimate in the one dimensional case and  for all $n \in (1,2)$ by relying on the fact that weak solutions
satisfy $u \in L^4((0,T); L^{\infty}(\Omega))$. Unfortunately, in the multi-dimensional case we only have that weak solutions are in $L^2((0,T); L^{\infty}(\Omega))$. This regularity is not
enough to control of some terms in the local entropy inequality when $n \in [\frac{s+2}{s+1},2) $. This explain why we impose that $n\in [1, \frac{s+2}{s+1})$.

In \cref{sec:FSP} and in \cref{sec:LWT}, we prove the FSP and the WTP, respectively.
The main line of our proof is inspired by the approach of
\cite{DalPassoGiacomelliGruen,GruenWTWS, GiacomelliGruen,MR1642176,CT12,MR2073864,MR2265292,MR3989405}. 
The proof of the finite speed of propagation is energetic in the sense that we prove various energy estimates that allow for the use of a Stampacchia-type iteration lemma. However, due to the nonlocal nature of the operator $(-\Delta)^s$, we cannot apply this strategy directly (indeed, some extra term appears in the local entropy estimate). 
Therefore, a suitable argument by contradiction is needed. 

In \cref{app:lemmas}, we collect some technical results used in the proofs, including the
statements of the above-mentioned Stampacchia's lemmas, a modified Gagliardo--Nirenberg interpolation inequality, and some estimates on the tail-behavior of the fractional Laplacian.

\section{Fractional Sobolev norms and the  fractional Laplacian}\label{PSP}

In this Section, we recall and state the results on Sobolev fractional norms and the spectral fractional Laplacian used in the paper.

Let $\Omega\subset \R^d$ be a bounded domain with smooth boundary.
Let $\lambda_k,\phi_k$ for $k=0,1,2,\ldots$, the sequence of the eigenvalues and of the normalized eigenfunctions of the Laplacian on $\Omega$
with homogeneous Neumann boundary conditions,
hence the unique solutions of the problems
\begin{equation}\label{eq:EVN}
\begin{cases}
- \Delta \phi_k = \lambda_k \phi_k  & \text{ in }   \Omega ,\\
\nabla \phi_k \cdot \textbf{n} = 0
 &  \text{ on } \partial\Omega,\\
 \|\phi_k\|_{L^2(\Omega)}=1.
\end{cases}
\end{equation}

We recall that $\lambda_0=0$ and $\phi_0=1/\sqrt{|\Omega|}$.
Moreover
\begin{equation}\label{EVlim}
	\lambda_0=0<\lambda_1\leqslant   \lambda_2\leqslant   \ldots, \qquad \lim_{n\to+\infty}\lambda_n=+\infty.
\end{equation}

For $u,v\in L^2(\Omega)$ we denote by $(u,v)=\int_\Omega u(x)v(x)\,\d x$ the scalar product in $L^2(\Omega)$.
We recall that $\{\phi_k\}_{k\geqslant  0}$ is an Hilbertian basis of $L^2(\Omega)$.
Given $u\in L^2(\Omega)$ we denote by
$$c_k(u)\coloneqq (u,\phi_k)$$ its Fourier coefficients.
We recall that Parseval's identity holds:
\begin{equation}\label{Parseval}
	\|u\|^2_{L^2(\Omega)}=\sum_{k=0}^{+\infty} |c_k(u)|^2, \qquad (u,v)=\sum_{k=0}^{+\infty} c_k(u)c_k(v).
\end{equation}
For $r\in [0,+\infty)$ we define the homogenous $\dot H_N^r(\Omega)$ seminorm by
\begin{equation}\label{defHr}
	\|u\|^2_{\dot H_N^r(\Omega)}\coloneqq  \mathop{\sum}_{k =0}^{+\infty} {\lambda_k^{r} |c_k(u)|^2}
\end{equation}
and the $H_N^r(\Omega)$ norm by
\begin{equation}
	\|u\|^2_{H_N^r(\Omega)}\coloneqq  \|u\|^2_{L^2(\Omega)} + \|u\|^2_{\dot H_N^r(\Omega)}.
\end{equation}
The Sobolev space $H_N^r(\Omega)$ is defined by
$$
H_N^r(\Omega)\coloneqq  \{u\in L^2(\Omega):\,  \|u\|^2_{\dot H_N^r(\Omega)}<+\infty\}.
$$
Notice that $\|u\|^2_{\dot H_N^0(\Omega)}= \|u\|^2_{L^2(\Omega)}$.
The space $H_N^r(\Omega)$ coincides with the classical fractional Sobolev space $H^r(\Omega)$ for $r\in (0,3/2)$. In the case $r\in (3/2,7/2)$, we have $H_N^r(\Omega)\coloneqq  \{u\in H^r(\Omega): \nabla u \cdot \textbf{n} = 0 \text{ on } \partial\Omega\}$. In any case, we have equality of the norms $\|u\|^2_{H_N^r(\Omega)}=\|u\|^2_{H^r(\Omega)}$ for any $u\in H_N^r(\Omega)$.
Moreover, using \cref{EVlim}, it is not difficult to show that
\begin{equation}\label{eq:equiv}
	\|u\|^2_{L^2(\Omega)} + \|u\|^2_{\dot H_N^r(\Omega)} \quad \text{is equivalent to}\quad \left(\int_\Omega u(x)\,\d x\right)^2 + \|u\|^2_{\dot H_N^r(\Omega)}.
\end{equation}

Using \cref{EVlim} it is immediate to prove that $ H^{r_1}_N(\Omega)\subset H^{r_2}_N(\Omega)$ if $r_1>r_2$. Moreover, the following embeddings hold (see \cite{MR2944369}).

\begin{proposition}
Let $r\in [0,+\infty)$. 
\begin{itemize}
	\item If $r<d/2$, then there exists a constant $C$ such that
	\begin{equation}\label{emb1}
 		\| u \|_{L^{2d/(d-2r)}(\Omega)}\leqslant   C  \| u \|_{H^r(\Omega)} \qquad\text{for all } u\in H^r(\Omega).
	\end{equation}
	\item If $r=d/2$ and $p\in (1,+\infty)$, then there exists a constant $C_p$ such that
	\begin{equation}\label{emb2}
 		\| u \|_{L^p(\Omega)}\leqslant   C_p  \| u \|_{H^r(\Omega)} \qquad\text{for all } u\in H^r(\Omega).
	\end{equation}
	\item If $r>d/2$ and $r-d/2\not\in\N$, then there exists a constant $C$ such that
	\begin{equation}\label{emb4}
 	\|u\|_{C^{r-d/2}(\Omega)} \leqslant   C  \| u \|_{H^r(\Omega)} \qquad\text{for all } u\in H^r(\Omega).
	\end{equation}
In particular, 
	\begin{equation}\label{emb3}
 	\| u \|_{L^\infty(\Omega)}\leqslant   C  \| u \|_{H^r(\Omega)} \qquad\text{for all } u\in H^r(\Omega).
	\end{equation}
\end{itemize}
\end{proposition}

The following interpolation inequality follows from the definition of $H^r(\Omega)$-norm and H\"{o}lder's inequality.
\begin{proposition}\label{prop:interp}
The following interpolation of the semi-norms hold.
If $r_0, r, r_1\in [0,+\infty)$, $r_0 \leqslant   r \leqslant   r_1$ and $u\in H^{r_1}_N(\Omega)$, then
\begin{equation}\label{interpsemi}
 	\| u \|_{\dot H^r_N(\Omega)}\leqslant   \| u \|^{1-\theta}_{\dot H^{r_0}_N(\Omega)}\| u \|^\theta_{\dot H^{r_1}_N(\Omega)},
\end{equation}
where $\theta \coloneqq \frac{r-r_0}{r_1-r_0}$.
\end{proposition}

\begin{proof}
The key observation is that applying H\"older's inequality  with $p = \frac{r_0}{r \nu }$ and $ q = \frac{r_1}{r (1-\nu ) }$ yields 
\[
\sum_{k=0}^{\infty}  \lambda_k^r \abs{c_k(u)}^2  = \sum_{k=0}^{\infty}  \lambda_k^{r \nu}  \lambda_k^{r (1-\nu) }   \abs{c_k(u)}^2 
\leqslant 
\left(\sum_{k=0}^{\infty}  \lambda_k^{r_0} \abs{c_k(u)}^2 \right)^{1-\theta}  \left(\sum_{k=0}^{\infty}  \lambda_k^{r_1} \abs{c_k(u)}^2 \right)^{\theta},
\]
where  $\nu \coloneqq \frac{ r_0 (r_1 - r) }{r (r_1 - r_0) }$.
\end{proof}

For $r\in [0,+\infty)$ and $u\in H^{2r}_N(\Omega)$ we define the $r$-Laplacian operator by
\begin{equation}\label{rlap}
(-\Delta)^r u \coloneqq   \mathop {\sum}_{k =0}^{+\infty} {  \lambda_k^{r} c_k(u) \phi_k}.
\end{equation}

For the fractional Laplacian the following identity for its $L^2(\Omega)$ norm holds.
\begin{proposition}\label{prop:2r}
If $u\in H^{2r}_N(\Omega)$, then
\begin{equation}\label{charL2}
 	\| (-\Delta)^r u \|^2_{L^2(\Omega)} = \| u \|^2_{\dot H^{2r}_N(\Omega)}.
\end{equation}
\end{proposition}
\begin{proof}
From \cref{rlap} we have that $c_k((-\Delta)^r u)\coloneqq ((-\Delta)^r u,\phi_k)= \lambda_k^{r} c_k(u)$. By Parseval identity \cref{Parseval} we have
$\| (-\Delta)^r u \|^2_{L^2(\Omega)} = \sum_{k=0}^{+\infty} \lambda_k^{2r}|c_k(u)|^2$ and we conclude applying \cref{defHr}.
\end{proof}

Analogously, using the Parseval product formula in \cref{Parseval}, it is simple to prove the following integration by parts formula.
\begin{proposition}\label{lem:ip}
Let $r_1, r_2 \in [0,+\infty)$. If $u \in H_N^{r_1+r_2}(\Omega)$ and $v \in H_N^{r_2}(\Omega)$, then
\begin{equation}\label{ip}
	 \int_\Omega((-\Delta)^{r_1} u) ((-\Delta)^{r_2} v)\,\d x =  \int_\Omega((-\Delta)^{r_1+r_2} u )\, v\,\d x.
\end{equation}
\end{proposition}

An immediate consequence of  \cref{prop:interp,prop:2r} is the following inequality
\begin{equation}\label{dong-01}
	\| (-\Delta)^{\beta} v \|_{L^2(\Omega)} \leqslant    \| (-\Delta)^{\frac{s+1}{2}} v \|_{L^2(\Omega)}^{\theta} \| v \|_{L^2(\Omega)}^{1-\theta},
\end{equation}
valid for $v \in   H^{s+1}_N(\Omega)$ and $\beta \in (0, \frac{1+s}{2})$,
where $ \theta \coloneqq \frac{2\beta}{s+1}$.

\section{Main results}\label{sec:main}

The mobility function $f:\R\to\R$, defined by $f(u)=u^n$ for $u\geqslant  0$, $f(u)=0$ for $u<0$,
corresponds to a natural entropy function $G_0:\R\to [0,+\infty]$ defined by the properties $G''_0(u)=1/u^n$ for any $u>0$ and $G'_0(1)=G_0(1)=0$.
A simple computation shows that, for any $u>0$,
\begin{equation}\label{eq:G}
G_0(u)=\left\{
\begin{array}{ll}
\displaystyle u \ln u - u +1 & \mbox{ if } n=1, \\[10pt]
\displaystyle  \frac{u^{2-n}}{(n-2)(n-1)} +\frac{u}{n-1}+ \frac{1}{2-n} & \mbox{ if } 1< n< 2, \\[10pt]
\displaystyle \ln \frac 1 u +u-1 & \mbox{ if } n=2,\\[8pt]
\displaystyle \frac{1}{(n-2)(n-1)}\frac{1}{u^{n-2}} +\frac{u}{n-1} -\frac{1}{n-2} & \mbox{ if } n> 2,
\end{array}
\right.
\end{equation}
and $G_0(u)=+\infty$ for $u<0$, $G_0(0)=\displaystyle\lim_{u\to 0^+}G_0(u)$.
We observe that $G_0$ is a non-negative and lower-semicontinuous convex function with a minimum point at $u=1$.

The entropy $G_0$ plays a pivotal role in the proof of our first main result, which deals with the existence of weak solutions for problem \cref{eq:ft}.

\begin{theorem}[Existence of weak solutions]\label{Th-ex}
 Let us assume that
 \begin{equation}\label{eq:ass_1}
s\in (0,1) \qquad \textrm{ and } \qquad n\in \left[1,\frac{d+2(1-s)}{(d-2s)_{+}}\right),
 \end{equation}
 with the convention that $\frac10=+\infty$.

If the initial datum $u_0$ satisfies
\begin{equation}\label{eq:initial}
 	u_0\in H^s(\Omega), \qquad u_0 \geqslant  0, \qquad  \int_{\Omega} G_0(u_0) \,\d  x < \infty,
\end{equation}
and $T>0$, then there exists a weak solution $u$ of the problem \cref{eq:ft} in the sense that
\begin{align*}
	u &\in L^\infty((0,T); H^s(\Omega))\cap L^2((0,T);H_N^{1+s}(\Omega))\cap C([0,T];L^2(\Omega)), \\
	\partial_t u &\in L^2((0,T); (W^{1,q}(\Omega))^*),
\end{align*}
for $q\coloneqq \frac{4d}{2d - n(d-2s)_+}\geqslant  2$
and
\begin{equation}
\label{eq:weak_sol}
\begin{aligned}
& -\iint_{\Omega_T}u\partial_t v \,\d x \,\d t =    \iint_{\Omega_T}  u^n p   \Delta v \,\d x \,\d t
+n   \iint_{\Omega_T} u^{n-1} p  \nabla u  \cdot \nabla v \,\d x \,\d t + \int_{\Omega}u_0v \,\d x, \\
&\qquad\qquad\text{for all }v\in C^{\infty}_{c}(\overline{\Omega}\times [0,T))
\textrm{ such that } \nabla v\cdot {\bf n}=0 \textrm{ on }\partial\Omega\times(0,T),\\
&p = (-\Delta )^s u\qquad \textrm{ a.\,e.~in } \Omega_T.
\end{aligned}
\end{equation}

Moreover, $u$ satisfies the following properties:
\begin{itemize}
	\item {non-negativity:}
	\begin{equation}
		u(x,t)\geqslant  0 \qquad \text{for all } (x,t) \in \Omega_T,
	\end{equation}
	\item {mass conservation:}
	\begin{eqnarray}\label{eq:masscons}
	\int_{\Omega} {u(x,t) \,\d x} = \int_{\Omega} {u_0(x) \,\d x} \qquad\text{for all } t\in (0,T],
	\end{eqnarray}
	\item {entropy dissipation:} \begin{eqnarray}\label{eq:entropy}
	\int_\Omega G_0(u(x,t)) \,\d x + \int_0^t{\| u(\cdot,r)\|_{\overset{.}{H}_N^{s+1}(\Omega)}^2 \,\d r \leqslant   \int_\Omega G_0(u_0) \,\d x}  \qquad\text{for all } t\in (0,T].
	\end{eqnarray}
\item {energy dissipation:}
  \begin{eqnarray}\label{eq:ineq}
	\| {u(\cdot,t)}\|_{\overset{.}{H}_N^{s}(\Omega)}^2 + 2 \iint_{\Omega_t} \abs{g(x,r)}^2 \,\d x \,\d r \leqslant   \| u_0 \|_{\overset{.}{H}_N^s(\Omega)}^2  \qquad\text{for all } t\in (0,T],
	\end{eqnarray}
	where the vector field $g\in L^2(\Omega_T;\R^d)$ is implicitly defined by means of 
	\begin{equation}
	\label{eq:pseudo_flux_weak}
	u^{n/2}g = \nabla \left(u^n p\right)-nu^{n-1}p\nabla u\qquad \text{a.\,e.~in }\Omega_T.
	\end{equation}
\end{itemize}
Moreover, if $n\geqslant 2$, then
\begin{equation}\label{positivityu}
{\mathscr{L}^d\left(\{ u(\cdot,t) = 0 \} \right) = 0 \quad \textrm{ for any }t>0,}
\end{equation}
{and we can identify $g$ as }
 \begin{equation}
 \label{eq:pseudo_flux}
  { \iint_{\Omega_T}g\cdot\phi \,\d x \,\d t = -\iint_{\Omega_T} { u^{\frac{n}{2}}p \div\phi \,\d x \, \d t}
- \frac{n}{2}\iint_{\Omega_T} { u^{\frac{n}{2}-1} p \nabla u \cdot \phi \,\d x \,\d t} }
\end{equation}
{ for all $ \phi \in C_c^\infty ({\Omega}_T;\R^d).$}

If $d=1$ or $d=2$ and $n> 2 + \frac{2d}{2(s+1)-d}$, then 
 \[
 	u(\cdot,t) \in C^{s+1-d/2}(\overline{\Omega}) \qquad \textrm{ for a.\,e.~} t \in (0,T),
 \]
and
 \begin{equation}\label{strictpositivity}
	 u(x,t)>0\qquad\text{for all }x\in \overline\Omega \qquad \text{ for a.\,e.~} t \in (0,T).
 \end{equation}
\end{theorem}

We observe that, for $d=1$, the condition \cref{eq:ass_1} amounts to
$n\in [1,+\infty)$. Analogously, observe that the exponent $q\coloneqq \frac{4d}{2d - n(d-2s)_+}>2$ in the case $s<d/2$ and $q=2$  in the case $s\geqslant  d/2$ only (hence, for $d=1$).

Then we concentrate on the propagation properties of the solutions constructed in \cref{Th-ex}.

Our first result in this direction is an upper bound on the speed of propagation of the solution support.

\begin{theorem}[Upper bounds on interface propagation speed]\label{th:speed}
	Let $R>0$ and assume $\Omega = B_R(0)$, $s \in (\frac{(d-2)_+}{2},1)$, $n \in (1,   \frac{s+2}{s+1}  )$, and
	\begin{equation}\label{e-11}
\supp u_0 \subseteq B_{r_0}(0)
	\end{equation}
	for some $0 < r_0 < R$. Let $u$ be a solution of \cref{eq:ft} given by \cref{Th-ex}. Then there exists a time $T^* >0$ and a nondecreasing function $d(t) \in C([0,T^*])$ such that
$$
	d(0) = r_0, \qquad d(t) \leqslant   r_0 + C_0 t^{\frac{1}{nd + 2(s+1)}} \qquad\text{for all } t \in [0,T^*],
$$
where $C_0 > 0$ depends on initial data $u_0$, $d$, $s$ and  $n$,
such that
	\begin{align*}
	\supp(u(\cdot,t)) \subset  B_{d(t)}(0) \subseteq \Omega \qquad\text{for all } t \in [0,T^*].
	\end{align*}
\end{theorem}

Finally, we obtain a lower bound on the waiting time.

\begin{theorem}[Lower bounds on waiting times]\label{th:wt}
Let us assume  $\Omega = B_R(0)$, $s \in (\frac{(d-2)_+}{2},1)$, $n \in (1,   \frac{s+2}{s+1}  )$,
\cref{e-11} for some $0 < r_0 < R$, and that
\begin{equation}\label{G}
\limsup_{\delta \to 0^+} \, \left( \delta^{-\gamma(2-n)} \fint_{B_{r_0}(0)\setminus B_{r_0-\delta}(0)} |G_0(u_0(x)) - G_0(0)| \dd x\right) < \infty
\end{equation}
for $\gamma \geqslant  \frac{2(s+1)}{n}$.  Let $u$ be a solution of \cref{eq:ft} given by \cref{Th-ex}. Then there exists a time $T_0 = T_0(n,s, u_0)$ such that
\begin{align*}
\supp(u(\cdot,t)) \subseteq B_{r_0}(0) \quad \text{ for } t \in [0,T_0).
\end{align*}
Moreover, the waiting time is estimated from below by
\begin{align*}
T_0 \geqslant  C \left( \sup_{\delta>0} \delta^{-\gamma(2-n)} \fint_{B_{r_0}(0)\setminus B_{r_0-\delta}(0)} | G(u_0(x)) -G_0(0)| \dd x \right)^{-\frac{n}{2-n}}.
\end{align*}
\end{theorem}

\section{Proof of the existence result}\label{sec:existence}

This section is devoted to the proof of  \cref{Th-ex}.

\subsection{A regularized problem}

Let us fix $\varepsilon, \delta, \gamma>0$, and
a parameter  $\alpha > \max \{2,n\}$. We consider the approximations of the mobility function $f(z)=(z)_+^n$, $z\in \R$, defined by
\begin{equation}\label{defregmob}\begin{aligned}
	f_{\eps,\delta}(z) &\coloneqq  \begin{cases} \displaystyle \frac{z^{n+\alpha}}{z^{\alpha}+\varepsilon z^{n}+\delta z^{n+\alpha}} & \text{if } z>0,\\ 
	 0 & \text{if }z\leqslant  0, 
    \end{cases} \\ 
	f_{\eps,\delta,\gamma}(z) &\coloneqq  f_{\eps,\delta}(z) + \gamma, \qquad \text{for all } z\in\R.
\end{aligned}\end{equation}

We observe that the parameter \(\delta>0\) yields the boundedness of  \(f_{\eps,\delta}\) and \(f_{\eps,\delta,\gamma}\), precisely it holds 
\begin{equation}\label{boundedf}
	0\leqslant   f_{\eps,\delta}(z) \leqslant   1/\delta, \qquad  \gamma \leqslant   f_{\eps,\delta,\gamma}(z) \leqslant   1/\delta +\gamma  \qquad\text{for all }z\in\R.
\end{equation}
The parameter \(\varepsilon>0\) has the effect of increasing the degeneracy at $0$ from \(f(z) \sim (z)_+^{n}\) to 
\begin{equation}\label{eq:f0}
	f_{\eps,\delta}(z) \sim \tfrac{1}{\eps}(z)_+^{\alpha} \qquad \text{as } z\to 0. 
\end{equation}
Moreover, since
$$f'_{\varepsilon,\delta}(z)=\frac{z^{n+\alpha -1}(n z^{\alpha} + \varepsilon \alpha z^n)}{(z^{\alpha} + \varepsilon z^n + \delta z^{n+\alpha})^2}
\quad \text{ if }z>0, \quad f'_{\eps,\delta}(z)=0 \quad \text{ if }z\leqslant  0,$$
we have
\begin{equation}
\label{eq:f'}
\begin{aligned}
f'_{\eps,\delta}(z) &\sim \frac{\alpha}{\eps}(z)_+^{\alpha-1} && \textrm{as }
{z}\to 0,\\
f'_{\eps,\delta}(z) &\sim \frac{n}{\delta^2}{z}^{-n-1} && \textrm{as } {z} \to +\infty.
\end{aligned}
\end{equation}
Since $\alpha>2$, from \cref{eq:f'} and the continuity of $f'_{\varepsilon,\delta}$ it is immediate to prove that there exists a constant
$L>0$, depending on $\varepsilon$ and $\delta$, such that $|f'_{\varepsilon,\delta}(z)| \leqslant   L$ for any $z\in\R$, i.\,e., $f_{\eps,\delta}$ is Lipschitz continuous.

We consider an approximation of the non-negative initial datum $u_{0}$ such that
\begin{align}
	&u_{0, \eps,\delta}\in H^1(\Omega), && u_{0, \eps,\delta} \geqslant  u_{0}+\varepsilon^{\theta_{1}}+\delta^{\theta_{2}}  && \text { for some } 0<\theta_{1}< \frac{1}{\alpha -2},\  \theta_{2}>0, \label{H1p} \\
&u_{0, \eps,\delta,\gamma} \in H^1(\Omega), && u_{0, \eps,\delta,\gamma} \to u_{0,\eps,\delta} && \text { strongly in } H^{1}(\Omega)  \text { as } \gamma \rightarrow 0,\label{H1c} \\
&{} &&u_{0, \eps,\delta} \to u_{0} && \text { strongly in } H^{s}(\Omega)  \text { as } \eps,\delta \rightarrow 0, \label{Hsc}
\end{align}
where the parameters \(\varepsilon\) and \(\delta\) are used to lift the initial data to be strictly positive even if \(\gamma=0\).

Let us consider the regularized problem of \cref{eq:ft} in $\Omega$:
\begin{equation}
\label{eq:ft-r}\tag{$P_{\eps,\delta,\gamma}$}
\begin{cases}
\partial_t u (x,t) = \div(f_{\eps,\delta,\gamma}(u(x,t)) \nabla p(x,t) ), & (x,t) \in \Omega_T,\\
p(x,t)= \Ds u(x,t), &  (x,t) \in \Omega_T,\\
\nabla u \cdot \textbf{n} = \nabla p \cdot \textbf{n} = 0, &  (x,t) \in  \partial \Omega  \times (0,T),\\
u(x,0) =u_{0,\eps,\delta,\gamma}(x), & x \in \Omega.
\end{cases}
\end{equation}

We shall first prove the existence of a sequence of weak solutions $u_{\eps,\delta,\gamma}$ of Problem \cref{eq:ft-r}; as a second step, we shall prove the compactness of such sequence by relying on suitable energy and entropy estimates; finally, we check that the limit function is a solution of \cref{eq:ft}.

We say that a couple $(u_{\eps,\delta,\gamma}, p_{\eps,\delta,\gamma})$ is a weak solution of \cref{eq:ft-r} if
\[
\begin{aligned}
&(u_{\eps,\delta,\gamma}, p_{\eps,\delta,\gamma} ) \in \left(L^{\infty}((0,T); H_N^s(\Omega)) \cap L^{2}((0,T); H_N^{2s+1}(\Omega))\right)
 \times L^{2}((0,T); H^{1}(\Omega)),\\
 &  \partial_t u_{\eps,\delta,\gamma} \in L^2((0,T); (H^{1}(\Omega))^*),\\
&\lim_{t\to 0^+}u_{\eps,\delta,\gamma}(\cdot,t) = u_{0,\eps,\delta,\gamma}(\cdot) \qquad \textrm{ a.\,e.~in } \Omega,
\end{aligned}
\]
and
\begin{equation}\label{apr-001}
\begin{cases}
\displaystyle\int_{0}^T  \langle \partial_t u_{\eps,\delta,\gamma} , v \rangle_{(H^1)^{*}, H^1} \, \d t & {}
\\ \displaystyle \qquad = -  \iint_{\Omega_T}
  f_{\eps,\delta,\gamma}(u_{\eps,\delta,\gamma} )  \nabla p_{\eps,\delta,\gamma} \cdot  \nabla v  \,\d x\, \d t & \text{for all } v \in L^{2}((0,T); H^{1}(\Omega)),\\
  \displaystyle p_{\eps,\delta,\gamma} = (-\Delta)^s u_{\eps,\delta,\gamma} & \textrm{ a.\,e.~in }
   \Omega_T.
   \end{cases}
\end{equation}

\subsection{The Faedo--Galerkin scheme}

To prove the existence of a solution to Problem \cref{eq:ft-r}
we rely on the Faedo--Galerkin approximation scheme. To ease the notation we remove the dependence on $\eps, \delta, \gamma$.

Using the notation of
\cref{PSP}, we denote by $\{\lambda_i \}_{i \in \mathbb{N}}$ and $\{\phi_i \}_{i \in \mathbb{N}}$ the sequences of the eigenvalues
and the normalized eigenfunctions of the Laplace operator with Neumann boundary conditions.
As already observed, the system of the eigenfunctions $\{\phi_i \}_{i \in \mathbb{N}}$ is an orthonormal basis of $L^2( \Omega)$ and an ortogonal system in $H^1(\Omega)$.

We introduce the discretized problem. For  $N\geqslant  0$ we define the $(N+1)$-dimensional space
\begin{equation}\label{eq:VN}
V_N \coloneqq  \textrm{span}\left\{\phi_i\right\}_{i=0,\ldots,N}
\end{equation}
and we observe that $V_N\subset H^{s}(\Omega)$ for every $s\in [0,1]$.
Moreover we have that $V_M\subset V_N$ for $N > M$ and that $\bigcup_{N=0}^{+\infty}V_N$ is dense in $H^s(\Omega)$ for any $s\in [0,1]$.

Defining
\begin{equation}\label{eq:proju0}
	u_{0}^N \coloneqq  \sum_{i=0}^{N}(u_{0,\eps,\delta,\gamma},\phi_i)\phi_i,
\end{equation}
the projection of the initial datum $u_{0,\eps,\delta,\gamma}$ on $V_N$,
we look for a couple $(u^N, p^N)\in H^1((0,T);V_N)\times L^2((0,T);V_N)$ solving  the system
\begin{equation}\label{eq:discretized}
\begin{cases}
\displaystyle
\int_{\Omega}  \partial_t u^N v\,\d x = - \int_{\Omega}  f_{\eps,\delta,\gamma}(u^N)  \nabla p^N \cdot \nabla v \,\d x,\\
\displaystyle\int_{\Omega}  p^N v \,\d x =   \int_{\Omega}  (-\Delta)^s u^N   v \,\d x,\\
\displaystyle u^{N}(\cdot,0) = u^N_{0}(\cdot) \quad \textrm{ a.\,e.~in } \Omega
\end{cases}
\end{equation}
for any $t\in (0,T)$ and for any $v \in V_N$. Taking $v=\phi_i$ for $i=0,\ldots,N$ in \cref{eq:discretized}, Problem \cref{eq:discretized} can be rewritten
as a system of nonlinear ordinary differential equations for the ''components'' of the vectors $u^N$ and $p^N$ in $V^N$ with respect to the chosen basis.
More precisely for $i=0,\ldots, N$ we look for functions $c_i\in H^1(0,T)$ and $d_i\in L^2(0,T)$ such that
\[
u^{N}(t) = \sum_{i =0}^{N} {c_i(t) \phi_i} \quad \text{and} \quad p^{N}(t) = \sum_{i =0}^{N} {d_i(t) \phi_i }
\]
solve \cref{eq:discretized}.

Using the spectral definition \cref{eq:fl} of the  operator $(-\Delta)^s$, it is immediate to show that
if
\[u^{N} = \sum_{i =0}^{N} {c_i \phi_i} \in V_N,
\]
then
\[(-\Delta)^su^{N} \in V_N \quad \text{and}\quad (-\Delta)^su^{N} = \sum_{i =0}^{N} \lambda_i^s{c_i \phi_i}.
\]
The second equation in \cref{eq:discretized} can be  rewritten as
\begin{equation}
\label{ap-02}
	d_j(t) = \lambda_j^s c_j(t), \qquad t\in(0,T), \quad j=0,\ldots,N.
\end{equation}
The first equation in  \cref{eq:discretized} and the initial datum condition are equivalent to the following
Cauchy problem for a system of ordinary differential equations for
$(c_0,\ldots,c_N)$:
\begin{equation}
\label{ap-01}
\begin{cases}
\displaystyle
\frac{\d}{\d t} c_j(t) = - \gamma \lambda_j^{s+1} c_j(t) & {}
\\ \displaystyle \qquad \qquad \qquad  - \sum_{k =1}^{N} \lambda_k^s c_k(t) \int_{\Omega}  f_{\eps,\delta } \Bigl( \sum_{i =1}^{N} {c_i(t) \phi_i} \Bigr)  \nabla \phi_k \cdot \nabla \phi_j  \,\d x, & \qquad  j = 0, \ldots, N,\\
c_j(0) = (u_0^N, \phi_j). & {}
\end{cases}
\end{equation}
Defining the function $F:\R^{N+1}\to\R^{N+1}$ by
$$(F(c_0,\ldots,c_N))_j = - \gamma \lambda_j^{s+1} c_j
- \sum_{k =1}^{N} \lambda_k^s c_k \int_{\Omega}  f_{\eps,\delta } \Bigl( \sum_{i =1}^{N} {c_i \phi_i} \Bigr)  \nabla \phi_k  \cdot\nabla \phi_j  \,\d x, \qquad j=0,\ldots,N,$$
it is not difficult to show that
$F$ is locally Lipschitz and there exists $L_N>0$ such that $|F(c)| \leqslant   L_N |c|$ for any $c\in\R^{N+1}$.
In particular, owing to Cauchy--Lipschitz/Picard-Lindel\"of's theorem, for any $T>0$, the problem \cref{ap-01} has a unique global classical solution $c\in C^1([0,T];\R^{N+1})$.

Now, we prove suitable uniform (with respect to $N$) estimates in order to pass to the limit as $N\to+\infty$.
First of all, we note that also at the discrete level the evolution preserves mass.
Indeed, taking $v = 1$ in \cref{eq:discretized} (recall that $\phi_0$ is constant and belongs to $V_N$), we get that
$$
\frac{\d}{\d t} \int_{\Omega} u^N\,\d x= 0,
$$
which readily implies that
\begin{equation}
\label{eq:mass_cons_galerkin}
\int_{\Omega}u^{N}(x,t)\,\d x = \int_{\Omega}u_0^{N}(x) \,\d x \qquad\text{for all } t\in [0,T].
\end{equation}
Multiplying the equation in \cref{ap-01}) by $\lambda_j^s c_j(t)$ and summing on $j$ from $1$ to $N$, we obtain that
\[
\frac{1}{2} \frac{\d}{\d t} \| u^N \|^2_{\dot{H}^s(\Omega)} + \gamma   \| u^N \|^2_{\dot{H}^{2s+1} (\Omega)}
+ \int_{\Omega} { f_{\eps,\delta}(u^N) | \nabla p^N|^2 \,\d x} = 0 \qquad \text{in }(0,T).
\]
This implies, since $f_{\eps,\delta}\geqslant  0$, that
\begin{equation}\label{ap-03}
\begin{aligned}
	&\frac{1}{2} \| u^N(t) \|^2_{\dot{H}^s(\Omega)} + \gamma  \int_0^t\| u^N(\tau) \|^2_{\dot{H}^{2s+1} (\Omega)} \,\d \tau
\leqslant   \frac{1}{2} \| u^N(0) \|^2_{\dot{H}^s(\Omega)}\leqslant   \frac{1}{2}\norm{u_{0,\eps,\delta,\gamma}}^2_{\dot{H}^s(\Omega)} 
\end{aligned}
\end{equation}
for all $t \in (0,T]$. Therefore, by \cref{eq:mass_cons_galerkin}, the equivalence \cref{eq:equiv} and the convergences \cref{H1c} \cref{Hsc}, we have that
 there exists a constant $C$ independent of $N$, $\eps$, $\delta$ and $\gamma$ such that
\begin{equation}\label{eq:estimate1_galerkin}
 	\| u^N \|_{L^{\infty}((0,T);{H}^s(\Omega))} \leqslant   C, \qquad  \| u^N\|_{L^2((0,T);{H}^{2s +1}(\Omega))} \leqslant   C/\gamma.
\end{equation}
Observing that  $ \| u^N(t)\|_{\dot{H}^{2s +1}(\Omega)}=  \| \nabla p^N(t)\|_{L^2(\Omega)}$,
the inequality \cref{eq:estimate1_galerkin} yields immediately that
\begin{equation}\label{eq:estpN}
  \|\nabla p^N\|_{L^2((0,T);{L}^{2}(\Omega))} \leqslant   C/\gamma.
\end{equation}

We fix  $w\in H^1(\Omega)$ such that $\norm{w}_{H^1(\Omega)}\leqslant   1$. Then $w$ decomposes as
$w = v + \xi$ with $v\in V_N$ and $\xi\perp V_N$ with respect to the scalar product in $L^2(\Omega)$, namely, $\int_{\Omega}\xi \phi_i\,\d x = 0$ for any $i=0,1,\ldots,N$.
Since  $\left\{\phi_i\right\}_{i\in\N}$ is also an orthogonal system in $H^1(\Omega)$, it is easy to show that $\norm{v}_{H^1(\Omega)}\leqslant   1$.
By \cref{eq:discretized} we have, for any $t\in(0,T)$,
\[
\langle \partial_t u^N,w\rangle = \int_{\Omega}\partial_t u^N w\,\d x = \int_{\Omega}\partial_t u^N v\,\d x = -\int_{\Omega}f_{\eps,\delta,\gamma}(u^N)\nabla p^N\cdot\nabla v\,\d x,
\]
\[
\langle \partial_t u^N,w\rangle = -\gamma \int_{\Omega}\nabla p^N\nabla v\,\d x- \int_{\Omega}f_{\eps,\delta}(u^N)\nabla p^N\cdot\nabla v\,\d x.
\]
Since  $\norm{w}_{H^1(\Omega)}\leqslant   1$ and $\norm{v}_{H^1(\Omega)}\leqslant   1$, the last equality implies that
\[
\abs{\langle \partial_t u^N,w\rangle}\leqslant   \gamma\norm{\nabla p^N}_{L^2(\Omega)} + \norm{f_{\eps,\delta}(u^N)\nabla p^N}_{L^2(\Omega)},
\]
and, by definition of dual norm,
\[
\norm{\partial_t u^N}_{\left(H^1(\Omega)\right)^*} \leqslant   \gamma\norm{\nabla p^N}_{L^2(\Omega)} + \norm{f_{\eps,\delta}(u^N)\nabla p^N}_{L^2(\Omega)}.
\]
By \cref{boundedf}, \cref{eq:estpN} and the last inequality, we obtain
\begin{equation}
\label{eq:estimate_2galerkin}
	\norm{\partial_t u^N}_{L^2((0,T);\left(H^1(\Omega)\right)^*)}\leqslant   (2\gamma+\frac{1}{\delta})\frac{C}{\gamma}
\end{equation}
for a constant $C$ independent of $N$.

Using well-known weak and weak-$\ast$ compactness results, from \cref{eq:estimate1_galerkin} and \cref{eq:estimate_2galerkin},
we obtain that there exist $u$ and a (not relabeled) subsequence of $N$ such that
\[
\begin{aligned}
& u^{N}\xrightarrow{N\to +\infty}u && \textrm{ weakly-$\ast$ in } L^{\infty}((0,T);H^s(\Omega))\\
& \partial_t u^N \xrightarrow{N\to +\infty}\partial_t u && \textrm{ weakly in }L^{2}((0,T); \left(H^1(\Omega)\right)^*),
\end{aligned}
\]
which implies, thanks to Aubin--Lions' compactness lemma (see, e.\,g., \cite{simon87}),
\begin{equation}\label{CSL2}
    u^N \xrightarrow{N\to +\infty}  u \qquad   \text{ strongly in } C([0,T]; L^2(\Omega)).
\end{equation}

In particular, possibly extracting a further subsequence, we get
\[
u^N \xrightarrow{N\to +\infty} u\qquad \textrm{almost everywhere in } \quad \Omega_T.
\]
This last convergence and Lebesgue's  dominated convergence theorem imply that
\[
f_{\eps,\delta,\gamma}(u^N)\xrightarrow{N\to +\infty}f_{\eps,\delta,\gamma}(u) \qquad\textrm{ strongly in } L^2((0,T);L^2(\Omega)).
\]
Finally, by \cref{eq:estimate1_galerkin} we have that
\[
  u^N \xrightarrow{N\to +\infty} u  \qquad    \text{ weakly in } L^{2}((0,T); H^{2s+1} (\Omega) ),
\]
and thus, using \cref{CSL2} and the interpolation \cref{interpsemi}
\[
u^N \xrightarrow{N\to +\infty} u \qquad  \text{ strongly in } L^2((0,T); H^{2s} (\Omega)).
\]
By construction we have that
\[
p^N = (-\Delta)^s u^N \qquad \textrm{  in } \Omega_T.
\]
Therefore, there exists $p\in L^2((0,T);H^1(\Omega))$ such that
\[
p^N\xrightarrow{N\to +\infty} p \qquad \textrm{ weakly in } L^2((0,T);H^1(\Omega)).
\]
In particular, it is easy to show that
\[
p = (-\Delta)^s u  \qquad \textrm{ almost everywhere in } \Omega_T.
\]
Now, we fix $M\geqslant  1$. Recall that $V_M\subset V_N$ for $N>M$.
Thus, for every $N>M$ and for every step function $v$ with values in $V_M$, we have that
\[
\begin{aligned}
&\displaystyle
\iint_{\Omega_T} \partial_t u^N v\,\d x\,\d t = - \iint_{\Omega_T}  f_{\eps,\delta,\gamma}(u^N)  \nabla p^N \cdot  \nabla v \,\d x\,\d t,\\
&\displaystyle \iint_{\Omega_T} p^N v \,\d x\,\d t =   \iint_{\Omega_T} v(-\Delta)^s u^N \,\d x\,\d t.
\end{aligned}
\]
Therefore, the above convergences readily imply that
 $u_{\eps,\delta,\gamma}\coloneqq u$ and $p_{\eps,\delta, \gamma}\coloneqq p$ verify
\begin{align}\label{eq:wf}
\iint_{\Omega_T} \partial_t  u_{\eps,\delta,\gamma}v\,\d x\,\d t  &= - \iint_{\Omega_T} f_{\eps,\delta,\gamma}(u_{\eps,\delta,\gamma})  \nabla p_{\eps,\delta, \gamma} \cdot \nabla v \,\d x\,\d t,\\
\iint_{\Omega_T} p_{\eps,\delta, \gamma} v \,\d x \,\d t &=   \iint_{\Omega_T}  v(-\Delta)^s u_{\eps,\delta,\gamma}\,\d x\,\d t\notag
\end{align}
for every step function $v$ with values in $V_M$, and thus for every test function with values in $\bigcup_{N=1}^{+\infty}V_N$.
Since this union is dense in $H^1(\Omega)$, we have that the above relations hold for any test function $v\in L^2((0,T);H^1(\Omega))$. This means that the couple
$(u_{\eps,\delta, \gamma}, p_{\eps,\delta,\gamma})$ is a solution of \cref{apr-001}.

Note that mass is still conserved. In fact, for any fixed $t\in(0,T]$ taking $v(x,\tau)=\chi_{[0,t]}(\tau)$ in the equation above we obtain
\begin{equation}
\label{eq:mass_cons_approx}
\int_{\Omega}u_{\eps,\delta,\gamma}(x,t)\,\d x = \int_{\Omega} u_{0,\eps,\delta,\gamma}(x) \,\d x \qquad\text{for all } t\in (0,T].
\end{equation}
Moreover, the couple $(u_{\eps,\delta, \gamma}, p_{\eps,\delta,\gamma})$ verifies the energy identity.
Indeed, for fixed $t\in (0,T]$ using $v(x,\tau)=p_{\eps,\delta,\gamma}(x,\tau)\chi_{[0,t]}(\tau)$
in the equation \cref{apr-001}, using Fubini's theorem, \cref{ip} and \cref{charL2}, we obtain
\begin{equation}
\label{eq:energy_esti1}
\norm{u_{\eps,\delta, \gamma}(\cdot,t)}^{2}_{\dot{H}^s(\Omega)}
+ 2 \iint_{\Omega_t}  f_{\eps,\delta,\gamma}(u_{\eps,\delta,\gamma}) \abs{\nabla p_{\eps,\delta,\gamma}}^2 \,\d x \,\d \tau
= \norm{u_{0,\eps,\delta,\gamma}}^2_{\dot{H}^s(\Omega)} \qquad\text{for all } t\in (0,T].
\end{equation}

\subsection{Limit for \texorpdfstring{$\gamma\to 0$}{gamma tending to zero}}

\subsubsection{Entropy estimate and its consequences}
\label{sssec:entropy}

Starting from the regularized mobility $f_{\eps,\delta,\gamma}$ defined in \cref{defregmob}, we define the positive function $G_{\eps,\delta,\gamma}:\R\to \R$ (the regularized entropy)
as the unique function satisfying
\begin{equation}\label{defGepsdeltagamma}
 	G''_{\eps,\delta,\gamma}(z)= \frac{1}{f_{\eps,\delta,\gamma}(z)} \qquad\text{for all }z\in \R, \qquad G'_{\eps,\delta,\gamma}(1)=G_{\eps,\delta,\gamma}(1)=0.
 \end{equation}
 Starting from $f_{\eps,\delta}$, we define
 the function $G_{\eps,\delta}:\R\to [0,+\infty]$ satisfying
$$
 	G''_{\eps,\delta}(z)= \frac{1}{f_{\eps,\delta}(z)} \qquad\text{for all } z\in (0,+\infty), \qquad G'_{\eps,\delta}(1)=G_{\eps,\delta}(1)=0,
$$
and
\begin{equation}\label{pos}
    G_{\eps,\delta}(z)=+\infty \qquad\text{for all } z\in (-\infty,0].
\end{equation}
A simple computation yields the following explicit expression:
\begin{equation}\label{diff_G}
	G_{\eps,\delta}(z)=G_0(z)+\frac{\eps}{\alpha-1}\bigl(\frac{z^{2-\alpha}}{\alpha-2} -\frac{1}{\alpha-2} +z-1\bigr)
	+ \frac{\delta}{2}(z^2 - 1 -2z+2) \qquad \text{for }z\in (0,+\infty).
\end{equation}

 From the previous definition, we deduce 
 \begin{equation}\label{convGedg}
	\lim_{\gamma\to 0}G_{\eps,\delta,\gamma}(z) = G_{\eps,\delta}(z)   \qquad \text{for all } z\in \R.
\end{equation}
 Consequently, using \cref{H1c} and \cref{H1p}, we obtain that
 \begin{equation}
 \label{eq:initial_entropy1}
 \lim_{\gamma\to 0}\int_{\Omega}G_{\eps,\delta,\gamma}(u_{0,\eps,\delta,\gamma})\,\d x = \int_{\Omega}G_{\eps,\delta}(u_{0,\eps,\delta})\,\d x.
 \end{equation}
Since $u_{\eps,\delta,\gamma}\in L^2((0,T);H^{2s+1}(\Omega))$, we have, in particular, that $u_{\eps,\delta,\gamma}\in L^2((0,T);H^{1}(\Omega))$.
 By \cref{boundedf} we have $G''_{\eps,\delta,\gamma}\in L^{\infty}(\R)$. Consequently $G'_{\eps,\delta,\gamma}$ is Lipschitz continuous and
 therefore $G'_{\eps,\delta,\gamma}(u_{\eps,\delta,\gamma})\in L^2((0,T);H^{1}(\Omega))$.
 For fixed $t\in(0,T]$ we can take $v(x,\tau) = G'_{\eps,\delta,\gamma}(u_{\eps,\delta,\gamma}(x,\tau)) \chi_{[0,t]}(\tau)$ in \cref{apr-001},
we obtain the entropy identity
\begin{equation}\label{ap-04}
\begin{aligned}
& \int_{\Omega}G_{\eps,\delta,\gamma}(u_{\eps,\delta,\gamma}(x,t)) \,\d x   +   \int_{0}^t \norm{u_{\eps,\delta,\gamma}}^2_{\dot H^{ s+1} (\Omega)} \,\d \tau =
 \int_{\Omega}G_{\eps,\delta,\gamma}(u_{0,\varepsilon,\delta,\gamma}) \,\d x
 \end{aligned}
\end{equation}
for all $t\in (0,T]$. By the entropy identity \cref{ap-04}, the conservation of mass \cref{eq:mass_cons_approx} and the equivalence \cref{eq:equiv}, and the convergence \cref{eq:initial_entropy1}, there exists a constant $C$ independent of $\gamma$ such that
\begin{equation}
	 \int_{0}^T \norm{u_{\eps,\delta,\gamma}}^2_{H^{ s+1} (\Omega)} \,\d t \leqslant   C.
\end{equation}
Analogously, from the energy identity \cref{eq:energy_esti1},  the conservation of mass \cref{eq:mass_cons_approx} and the equivalence \cref{eq:equiv}, the convergence \cref{H1c}, \cref{Hsc},
there exists a constant $C$ independent of $\eps$, $\delta$, and $\gamma$ such that
\begin{equation}\label{Hsbound}
	 \norm{u_{\eps,\delta,\gamma}(\cdot,t)}^2_{H^{s} (\Omega)} \leqslant   C \qquad\text{for all } t\in [0,T].
\end{equation}

By the above estimates and observing that
$$\norm{u_{\eps,\delta,\gamma}}^2_{\dot H^{ s+1} (\Omega)}= \norm{p_{\eps,\delta,\gamma}}^2_{\dot H^{ 1-s} (\Omega)},$$
there exist $u_{\eps,\delta}$ and $p_{\eps,\delta}$ and a subsequence of $\gamma$ (not relabelled) such that
\begin{align}
& u_{\eps,\delta,\gamma}\xrightarrow{\gamma \to 0}u_{\eps,\delta} \qquad \textrm{ weakly-$\ast$ in } L^{\infty}((0,T);H^{s}(\Omega)),\\
& u_{\eps,\delta,\gamma}\xrightarrow{\gamma\to 0}u_{\eps,\delta}\qquad \textrm{ weakly in } L^2((0,T);H^{1+s}(\Omega)),\label{Hconv} \\
& p_{\eps,\delta,\gamma}\xrightarrow{\gamma\to 0} p_{\eps,\delta}\qquad \textrm{ weakly in } L^2((0,T); H^{1-s}(\Omega))\label{eq:convpgamma}.
\end{align}
In particular, we have that
\[
p_{\eps,\delta} = (-\Delta)^s u_{\eps,\delta}\qquad \textrm{almost everywhere in } \Omega_T.
\]
As a consequence of \cref{eq:convpgamma}, we also have 
\begin{align}
\label{eq:convpgamma-r}	p_{\eps,\delta}&\xrightarrow{\eps,\delta\to 0}p\qquad \textrm{strongly in } L^2((0,T);H^{1-r}(\Omega)) \text{ for all } r\in (s,1],
\end{align}
and, having fixed $r\in(s,1)$, we can use \cref{p1conv} and the embedding \cref{emb1} to deduce
\begin{equation}\label{p1convga2}
 	p_{\eps,\delta}\xrightarrow{\eps,\delta\to 0}p\qquad \textrm{strongly in } L^2((0,T);L^{q'}(\Omega)),
\end{equation}
where $q'=\frac{2d}{d+2(r-1)}>2$.

Moreover, the energy estimate \cref{eq:energy_esti1} gives that
\begin{align}\label{boundflux--2}
\sqrt{f_{\eps,\delta,\gamma}(u_{\eps,\delta,\gamma})}\nabla p_{\eps,\delta,\gamma} \qquad \textrm{ is uniformly bounded w.\,r.\,t.~}\eps, \delta, \gamma \textrm{ in } L^2(\Omega_T).
\end{align}
Consequently, since $\gamma \leqslant   {f_{\eps,\delta,\gamma}(u_{\eps,\delta,\gamma})}\leqslant   \gamma + \frac{1}{\delta}$, we get that
\begin{equation}\label{boundfluxL2}
	f_{\eps,\delta,\gamma}(u_{\eps,\delta,\gamma})\nabla p_{\eps,\delta,\gamma} \qquad \textrm{ is uniformly bounded w.\,r.\,t.~} \gamma \textrm{ in } L^2(\Omega_T) \quad \textrm{for } \gamma \leqslant  1.
\end{equation}
We stress that the above estimate is not uniform in $\delta$.
Then, using the equation \cref{apr-001}, we get
\begin{equation}\label{bounddtugamma}
	\partial_t u_{\eps,\delta,\gamma} \qquad \textrm{ is uniformly bounded w.\,r.\,t.~}\gamma \textrm{ in } L^2((0,T); (H^{1}(\Omega))^*).
\end{equation}
By Aubin--Lions' compactness lemma (see  \cite{simon87}), we obtain that
\begin{align}
\label{eq:stronguepsdelta}
u_{\eps,\delta,\gamma}&\xrightarrow{\gamma\to 0}u_{\eps,\delta} \qquad \textrm{ strongly in } L^2((0,T); H^{1+s-\eta}(\Omega)) \text{ for all }\eta\in (0,s],
\\ 
\label{strongC}
u_{\eps,\delta,\gamma}&\xrightarrow{\gamma\to 0}u_{\eps,\delta} \qquad   \text{ strongly in  $C((0,T) ; L^p(\Omega) )$  for all   $p < \frac{2d} {d - 2s}$ }.
\end{align}
In particular, 
\begin{align}
\label{eq:ae-g-u}
 u_{\eps,\delta,\gamma}&\xrightarrow{\gamma\to 0}u_{\eps,\delta} && \textrm{ a.\,e.~in } \Omega_T, \\
\label{eq:ae-g-gradu}
\nabla u_{\eps,\delta}&\xrightarrow{\gamma\to 0}\nabla u && \textrm{ in }  L^2(\Omega_T)  \textrm{ and a.\,e.~in } \Omega_T.
\end{align}

Using  \cref{strongC}, we deduce that,  for all $t \in (0,T)$,
\begin{equation}\label{eq:convfgamma}
 			f_{\eps,\delta,\gamma}(u_{\eps,\delta,\gamma}(\cdot,t)) \xrightarrow{\gamma\to 0} f_{\eps,\delta}(u_{\eps,\delta}(\cdot,t)) \qquad \textrm{strongly in } L^{2}(\Omega).
\end{equation}

Using \cref{strongC}, we can also pass to the limit in \cref{eq:mass_cons_approx} obtaining the conservation of mass
\begin{equation}
\label{mass_cons2}
\int_{\Omega}u_{\eps,\delta}(x,t)\,\d x = \int_{\Omega} u_{0,\eps,\delta}(x) \,\d x \qquad\text{for all }t\in (0,T].
\end{equation}

We show that $u_{\eps,\delta}$ satisfies the entropy inequality.
First of all, we observe that if $0<\gamma_1<\gamma_2$ then
$G_{\eps,\delta}(z) > G_{\eps,\delta, \gamma_1}(z) >G_{\eps,\delta, \gamma_2}(z)$ for any $z\in \R$.
Then, for any fixed $\tilde\gamma>0$, by Fatou's lemma, the previous inequality, and \cref{strongC}, we have
\begin{align*}
\liminf_{\gamma\to 0}\int_{\Omega}G_{\eps,\delta, \gamma}(u_{\eps,\delta,\gamma}(x,t))\,\d x
 &\geqslant  \liminf_{\gamma\to 0}\int_{\Omega}G_{\eps,\delta,\tilde\gamma}(u_{\eps,\delta,\gamma}(x,t))\,\d x \geqslant  \int_{\Omega}G_{\eps,\delta,\tilde\gamma}(u_{\eps,\delta}(x,t))\,\d x
\end{align*}
for all $t\in (0,T]$. Using the monotone convergence theorem and the previous inequality, we obtain that
\[
\begin{aligned}
	&\liminf_{\gamma\to 0}\int_{\Omega}G_{\eps,\delta, \gamma}(u_{\eps,\delta,\gamma}(x,t))\,\d x
\geqslant  \int_{\Omega}G_{\eps,\delta}(u_{\eps,\delta}(x,t))\,\d x  \qquad\text{for all } t\in (0,T].
 \end{aligned}
\]
Using the last inequality, the semi-continuity of norms with respect to weak convergence \cref{Hconv},
the convergence \cref{eq:initial_entropy1}, from \cref{ap-04} we obtain that
\begin{equation}
\label{eq:entropy_approx2}
\begin{aligned}
 &\int_{\Omega}G_{\eps,\delta}(u_{\eps,\delta}(x,t)) \,\d x   +   \int_{0}^t \norm{u_{\eps,\delta}(\cdot,r)}^2_{\dot H^{ s+1} (\Omega)} \,\d r \leqslant  
 \int_{\Omega}G_{\eps,\delta}(u_{0,\eps,\delta}(x)) \,\d x 
 \end{aligned}
\end{equation}
for all $t\in (0,T]$.
Using the inequality \cref{eq:entropy_approx2} and \cref{pos}, we can show that 
\begin{equation}
\label{eq:positivity_approx}
\qquad u_{\eps,\delta}(x,t) \geqslant  0 \qquad \textrm{   in } {\Omega_T}
\end{equation}
(we refer to \cite[Proof of Theorem 1.2]{Grun95} for further details on this argument).
In particular, by \cref{eq:entropy_approx2},  recalling \cref{diff_G}, we have 
\[
\int_{\Omega} u_{\eps,\delta}^{2-\alpha} (x,t) \, \mathrm dx  \leqslant C \quad \text{for all } t\in (0,T], \quad \text{with } \alpha > 2.
\]
Then 
\begin{equation}
\label{eq:positivity_approx_bis}
\mathscr{L}^d\left(\{ x\in \Omega: u_{\eps, \delta} (x,t) = 0 \}\right) = 0  \quad \text{for all } t\in (0,T].
\end{equation}

\subsubsection{Limit \texorpdfstring{$\gamma \to 0$}{for gamma tending to zero} in the weak formulation}

In order to pass to the limit in the weak formulation \cref{apr-001} of the equation, we fix
$v\in C^{\infty}_c(\bar{\Omega}_T)$ such that $\nabla v\cdot {\bf n} =0 $ on $(0,T)\times \partial\Omega$ and rewrite the right- and left-hand sides as follows: 
\begin{align}\label{nlpartg}
	-\iint_{\Omega_T}{ f_{\eps,\delta,\gamma }(u_{\eps,\delta,\gamma } )  \nabla p_{\eps,\delta,\gamma }   \cdot\nabla v  \,\d x\, \d t}
	&= \iint_{\Omega_T}p_{\eps,\delta,\gamma}\nabla\left(f_{\eps,\delta,\gamma}(u_{\eps,\delta,\gamma})\right)\cdot\nabla v\,\d x\,\d t
 \\ & \qquad + \iint_{\Omega_T}f_{\eps,\delta,\gamma}(u_{\eps,\delta,\gamma})p_{\eps,\delta,\gamma}\Delta v\,\d x\,\d t, \notag
\\ 
\label{lpartg}
	\int \limits_{0}^T { \langle \partial_t u_{\eps,\delta } , v \rangle_{H^1(\Omega)^*, H^1(\Omega)} \,\d x}
	&= -\iint_{\Omega_T} u_{\eps,\delta,\gamma}\partial_t v\,\d x\,\d t -\int_{\Omega}u_{0,\eps,\delta,\gamma}(x)v(x,0)\,\d x.
\end{align}

We can pass to the limit in the right-hand side of \cref{nlpartg} using \cref{p1convga2}, \crefrange{eq:ae-g-u}{eq:convfgamma}. On the other hand, we can pass to the limit in \cref{lpartg} using \cref{strongC}. In conclusion, we have 
\begin{align}\label{apr-002a}
\begin{aligned}
	\int_{0}^T { \langle \partial_t u_{\eps,\delta } , v \rangle_{H^1(\Omega)^*, H^1(\Omega)} \,\d t}
		&= \iint_{\Omega_T}J_{\varepsilon,\delta}(u_{\varepsilon, \delta})\cdot {\nabla v} \,\d x\,\d t,
 \end{aligned}
\end{align}
{where $J_{\eps,\delta}$ is defined weakly as 
\begin{align}\label{apr-002b}
\begin{aligned}
	\iint_{\Omega_T}J_{\varepsilon,\delta}(u_{\varepsilon, \delta})\cdot V \,\d x\,\d t
		&= -\iint_{\Omega_T}p_{\eps,\delta}\nabla\left(f_{\eps,\delta}(u_{\eps,\delta})\right)\cdot V\,\d x\,\d t
 \\ & \qquad - \iint_{\Omega_T}f_{\eps,\delta}(u_{\eps,\delta})p_{\eps,\delta}\,\textrm{div} V\,\d x\,\d t
 \end{aligned}
\end{align}
for all $V \in L^{2}((0,T); H^{1}(\Omega;\R^d))$, and
$$
	  p_{\eps,\delta }(\cdot,t)  =  (-\Delta)^s u_{\eps,\delta} (\cdot,t)  \qquad \text{for a.\,e.~$t\in (0,T)$}.
$$
}

\subsubsection{Limit \texorpdfstring{$\gamma \to 0$}{for gamma tending to zero} in the energy identity and identification of the flux $J_{\eps,\delta}$}
{
In this Subsection we pass to the limit in the energy identity \cref{eq:energy_esti1} and we identify the flux $J_{\eps,\delta}$ in terms of $u_{\eps,\delta}$ and $p_{\eps,\delta}$ in the region where $u_{\eps,\delta}>0$. 

First of all, thanks to \cref{boundflux--2}, we deduce that 
\begin{equation}
\label{eq:energy_esti2a}
	\norm{u_{\eps,\delta}(\cdot,t)}^{2}_{\dot{H}^s(\Omega)}
	+ 2 \iint_{\Omega_t}  g_{\varepsilon, \delta}^2 \,\d x \,\d \tau
	\leqslant   \norm{u_{0,\eps,\delta}}^2_{\dot{H}^s(\Omega)} \qquad\text{for all } \,t\in (0,T],
\end{equation}
where  $g_{\eps,\delta}\in L^2(\Omega_T;\R^d)$ satisfies
\begin{equation}\label{u1conv4g}
 	f^{1/2}_{\eps,\delta,\gamma}(u_{\eps,\delta,\gamma})\nabla p_{\eps,\delta,\gamma}\xrightarrow{\gamma\to 0} g_{\eps,\delta} \qquad \textrm{weakly in } L^2((0,T);L^{2}(\Omega;\R^d)).
\end{equation}
 
 The definition of $f_{\eps,\delta,\gamma}$ yields  
 \[
 \iint_{\Omega_T} {f_{\eps,\delta, \gamma}(u_{\eps,\delta,\gamma}) |\nabla p_{\eps,\delta,\gamma}|^2 \, \mathrm dx \, \mathrm dt  } =  
 \underbrace{\gamma \iint_{\Omega_T}  |\nabla p_{\eps,\delta,\gamma}|^2 \, \mathrm  dx \, \mathrm dt  }_{\geqslant 0} +
  \iint_{\Omega_T} {f_{\eps,\delta}(u_{\eps,\delta,\gamma}) |\nabla p_{\eps,\delta,\gamma}|^2  \, \mathrm dx \, \mathrm dt  }.
 \]
 
 Now we work on the second term. To this end, we fix $\phi \in C_c^\infty (\bar{\Omega}_T;\R^d)$ and argue as follows. For $\mu>0$, we set $P_{\mu}\coloneqq \left\{(x,t)\in \Omega_T: \, u_{\eps,\delta}(x,t)>\mu\right\}$ and $N_{\mu}\coloneqq \Omega_T\setminus P_{\mu}$.
Since $\mathscr{L}^{d}\left(\{x\in \Omega: u_{\eps, \delta}(x,\cdot) = 0 \}\right)  = 0$ for all~$t > 0$ (which implies that $\mathscr{L}^{d+1}\left(\{ (x,t)\in \Omega_T:u_{\eps, \delta}(x,t) = 0 \}\right)=0$),
 we can choose the family $P_{\mu}$ to be monotonically increasing with respect to $\mu$. 
 More precisely, we can select a sequence $\mu_n\searrow 0$ when $n\to +\infty$ such that $\mu_{n+1}\le \mu_{n}$ 
 and $P_{\mu_{n}}\subseteq P_{\mu_{n+1}}$. Note that 
 \begin{equation}
 \label{eq:positivity_set}
 \left\{(x,t)\in \Omega_T: u_{\eps,\delta}(x,t)>0\right\} = \bigcup_{n=1}^{+\infty}\left\{(x,t)\in \Omega_T: u_{\eps,\delta}(x,t)>\mu_n\right\}. 
 \end{equation}
 As a result 
 \[
 \mathscr{L}^{d+1}\left(\left\{(x,t)\in \Omega_T: u_{\eps,\delta}(x,t)>0\right\}\right)=\lim_{n\to +\infty}\mathscr{L}^{d+1}\left(\left\{(x,t)\in \Omega_T: u_{\eps,\delta}(x,t)>\mu_n\right\}\right).
 \]
 Thus, since we have that $\mathscr{L}^{d+1}\left(\left\{(x,t)\in \Omega_T: u_{\eps,\delta}=0\right\}\right)=0$, we readily get 
\begin{equation}
\label{eq:positivity_small}
\lim_{n\to +\infty}\mathscr{L}^{d+1}(N_{\mu_n})=0.
\end{equation} 
We have that $u_{\eps,\delta,\gamma}\xrightarrow{\gamma\to 0}u_{\eps,\delta}$ almost everywhere in $\Omega_T$ (and, a fortiori, in $P_{\mu_n}$). 
Therefore, owing to Severini--Egorov's theorem,  there exists a monotonically increasing sequence of compact sets $K_{\lambda}\subset P_{\mu_n}$ such that 
\begin{enumerate}
\item $\abs{P_{\mu_n}\setminus K_\lambda}\le \lambda$;
\item $K_{\lambda}\subset K_{\lambda'}$ for $0<\lambda'<\lambda$;
\item $u_{\eps,\delta,\gamma}\xrightarrow{\gamma\to 0} u_{\varepsilon,\delta}$ uniformly in $K_{\lambda}$;
\item $u_{\eps,\delta,\gamma}\ge \mu_n$ in $K_{\lambda}$.
\end{enumerate}  
Therefore, as $z \mapsto f_{\eps,\delta}(z)$ is monotonically increasing, we have that $f_{\eps,\delta}(u_{\varepsilon,\delta})\ge f_{\eps,\delta}(\mu_n)$ on $K_{\lambda}$ and thus
\[
f_{\eps,\delta}(\mu_n)\iint_{K_{\lambda}}\abs{\nabla p_{\varepsilon,\delta,\gamma}}^2 \, \d x \, \d t \le C, \qquad \textrm{ uniformly in } \gamma. 
\]

As a result, we have that 
\[
\nabla p_{\eps,\delta,\gamma}\xrightarrow{\gamma \to 0}\nabla p_{\eps,\delta} \qquad \textrm{ weakly in } L^2(K_{\lambda}).
\]
Summing up,

\begin{align*}
&\lim_{\gamma\to 0}\iint_{P_{\mu_n}}f_{\eps,\delta}^{1/2}(u_{\varepsilon,\delta,\gamma})\nabla p_{\varepsilon,\delta,\gamma}\cdot\phi\, \d x \, \d t \\ & = \lim_{\gamma\to 0}\left(\iint_{K_{\lambda}}f_{\eps,\delta}^{1/2}(u_{\varepsilon,\delta,\gamma})\nabla p_{\varepsilon,\delta,\gamma}\cdot\phi\, \d x \, \d t
+\iint_{P_{\mu_n}\setminus K_\lambda}f_{\eps,\delta}^{1/2}(u_{\varepsilon,\delta,\gamma})\nabla p_{\varepsilon,\delta,\gamma}\cdot\phi\, \d x \, \d t\right)\\
& = \iint_{K_{\lambda}}f_{\eps,\delta}^{1/2}(u_{\eps,\delta})\nabla p_{\eps,\delta}\cdot\phi\d x\d t + \iint_{P_{\mu_n}\setminus K_\lambda}g_{\eps,\delta}\cdot\phi\, \d x \, \d t.
\end{align*}

As the sequence $\lambda\mapsto K_\lambda$ is monotone increasing, we use the Monotone Convergence Theorem 
\footnote{~To be precise, we apply the monotone convergence theorem by working separately on the set of those  $(x,t)\in P_{\mu_n}$ 
for which $\nabla p_{\eps,\delta}\cdot \phi\le 0$ and on the set in which $\nabla p_{\eps,\delta}\cdot\phi > 0$. }
(w.\,r.\,t.~$\lambda$) and we obtain 
\[
\lim_{\lambda\to 0}\int_{K_\lambda}f_{\eps,\delta}^{1/2}(u_{\eps,\delta})\nabla p_{\eps,\delta}\cdot\phi\d x\d t = 
\int_{P_{\mu_n}}f_{\eps,\delta}^{1/2}(u_{\eps,\delta})\nabla p_{\eps,\delta}\cdot\phi\d x\d t.
\]
On the other hand,
\[
\lim_{\lambda\to 0}\iint_{P_{\mu_n}\setminus K_\lambda}g_{\eps,\delta}\cdot\phi\, \d x \, \d t=0. 
\]
Therefore,
\[
\lim_{\gamma\to 0}\iint_{P_{\mu_n}}f_{\varepsilon,\delta,\gamma}^{1/2}(u_{\varepsilon,\delta,\gamma})\nabla p_{\varepsilon,\delta,\gamma}\cdot\phi\, \d x \, \d t = \iint_{P_{\mu_n}}f_{\eps,\delta}^{1/2}(u_{\eps,\delta})\nabla p_{\eps,\delta}\cdot\phi\,\d x\,  \d t. 
\]
In the region $N_{\mu_n}$, we estimate 
\begin{align*}
 \iint_{  N_{\mu_n} }  f^{1/2}_{\eps,\delta}(u_{\eps,\delta,\gamma}) |\nabla p_{\eps,\delta,\gamma}  \cdot \phi | \, \mathrm dx \, \mathrm dt  & \leqslant   
\left(  \iint_{ \Omega_T}  f_{\eps,\delta}(u_{\eps,\delta,\gamma}) | \nabla p_{\eps,\delta,\gamma} |^2   \, \mathrm dx \, \mathrm dt \right) ^{1/2}
 \left(\iint_{  N_{\mu_n} }  \phi^2 \, \mathrm dx \, \mathrm dt \right)^{1/2}  \\ &\leqslant 
 C \sqrt{\mathscr{L}^{d+1}(N_{\mu_n})}.
\end{align*}
Thus, we conclude that 
\[
\lim_{\gamma \to 0}\iint_{\Omega_T}f_{\varepsilon,\delta,\gamma}^{1/2}(u_{\varepsilon,\delta,\gamma})\nabla p_{\varepsilon,\delta,\gamma}\cdot\phi\, \d x \, \d t = \iint_{P_{\mu_n}}f_{\eps,\delta}^{1/2}(u_{\eps,\delta})\nabla p_{\eps,\delta}\cdot\phi\, \d x \, \d t + {o_{n\to +\infty}(1)}.
\]
Now, arguing again separately on the sets where $\nabla p_{\eps,\delta}\cdot\phi$ is positive and where it is negative, the Monotone convergence Theorem gives that 
\[
\lim_{n\to +\infty}\iint_{P_{\mu_n}}f_{\eps,\delta}^{1/2}(u_{\eps,\delta})\nabla p_{\eps,\delta}\cdot\phi\, \d x \, \d t = \iint_{\left\{u_{\eps,\delta}>0\right\}}f_{\eps,\delta}^{1/2}(u_{\eps,\delta})\nabla p_{\eps,\delta}\cdot\phi\, \d x \, \d t.
\] 
Therefore,
\[
\lim_{\gamma \to 0}\iint_{\Omega_T}f_{\eps,\delta}^{1/2}(u_{\eps,\delta,\gamma})\nabla p_{\varepsilon,\delta,\gamma}\cdot\phi\,  \d x \, \d t = \iint_{\left\{u_{\eps,\delta}>0\right\}}f_{\eps,\delta}^{1/2}(u_{\eps,\delta})\nabla p_{\eps,\delta}\cdot\phi\, \d x \, \d t.
\]
In conclusion, we can identify 
\[
g_{\eps,\delta}=
\begin{cases}
f_{\eps,\delta}^{1/2}(u_{\eps,\delta})\nabla p_{\eps,\delta} &\qquad \textrm{ on } \,\left\{(x,t) \in \Omega_T: \, u_{\eps,\delta}(x,t) >0\right\},\\
0&\qquad \textrm{ on } \left\{(x,t) \in \Omega_T:\, u_{\eps,\delta}(x,t)=0\right\}.
\end{cases}
\]

As a result \cref{eq:energy_esti2a} becomes 
\begin{equation}
\label{eq:energy_esti2}
	\norm{u_{\eps,\delta}(\cdot,t)}^{2}_{\dot{H}^s(\Omega)}
	+ 2 \int_0^t\int_{\{x\in\Omega:u_{\eps,\delta}(x,\tau) >0\}}  f_{\eps,\delta}(u_{\eps,\delta}) \abs{\nabla p_{\eps,\delta}}^2 \,\d x \,\d \tau
	\leqslant   \norm{u_{0,\eps,\delta}}^2_{\dot{H}^s(\Omega)} 
\end{equation}
for all  $t\in (0,T]$.
A similar argument allows to obtain informations on the flux $J_{\eps,\delta}$ and finally get 
\begin{equation}
\label{eq:flux_approx}
 J_{\eps, \delta}(u_{\eps,\delta}) \coloneqq 
\begin{cases}
 f_{\eps,\delta} (u_{\eps, \delta} ) \nabla p_{\eps, \delta} & \text{ in } \{ u_{\eps, \delta}  > 0 \},\\
0 & \text{ in } \{  u_{\eps, \delta} = 0 \}.
\end{cases} 
\end{equation}
}

Therefore, we can write the weak formulation as
\begin{equation}\label{apr-002}
	{\int_{0}^T { \langle \partial_t u_{\eps,\delta } , v \rangle_{H^1(\Omega)^*, H^1(\Omega)} \,\d t}}
	{= - \iint_{\{u_{\eps,\delta}>0\}}{ f_{\eps,\delta }(u_{\eps,\delta } )  \nabla p_{\eps,\delta }   \cdot\nabla v  \,\d x\, \d t}}
\end{equation}
for all $v\in C^{\infty}_c({\overline\Omega}\times [0,T))$ such that $\nabla v\cdot {\bf n} =0 $ on $(0,T)\times \partial\Omega$.

\subsection{Limits for \texorpdfstring{$\eps, \delta\to 0$}{epsilon and delta tending to zero} }

\subsubsection{A priori estimates and compactness}

Recalling \cref{diff_G}, we get
\[
G_{\varepsilon,\delta} (z) - G_{0} (z) = \frac{\varepsilon}{(\alpha -1)(\alpha-2)}(z^{2-\alpha} -1 ) +
(\frac{\varepsilon}{ \alpha -1 } - \delta )(z-1) + \frac{\delta}{2}(z^2 -1),
\]
and, using \cref{H1p},
\[
\int_{\Omega} { |G_{\varepsilon,\delta} (u_{0,\varepsilon,\delta}) -
 G_{0} (u_{0,\varepsilon,\delta}) | \,\d x } \leqslant   C \varepsilon^{1 - \theta_1(\alpha -2)}.
\]
Using \cref{Hsc}, we obtain that
\begin{equation}\label{convid}
\lim_{(\eps,\delta)\to (0,0)}\int_{\Omega} {G_{\varepsilon,\delta} (u_{0,\varepsilon,\delta})  \,\d x }   =
\int_{\Omega} {  G_{0} (u_{0})  \,\d x }.
\end{equation}

By the entropy estimate \cref{eq:entropy_approx2}, the conservation of mass \cref{mass_cons2} and the equivalence \cref{eq:equiv},
and the convergence \cref{convid}, there exists a constant $C$ independent of $\eps$, $\delta$
such that
\begin{equation}\label{boundHsplus1}
	 \int_{0}^T \norm{u_{\eps,\delta}(\cdot,t)}^2_{H^{ s+1} (\Omega)} \,\d t \leqslant   C.
\end{equation}
Analogously, from the energy identity \cref{eq:energy_esti1},  the conservation of mass \cref{eq:mass_cons_approx} and the equivalence \cref{eq:equiv}, and the convergence
 \cref{Hsc},
there exists a constant $C$ independent of $\eps$, $\delta$ such that
\begin{equation}\label{boundHsepsdelta}
	 \norm{u_{\eps,\delta}(\cdot,t)}^2_{H^{s} (\Omega)} \leqslant   C \qquad\text{for all } t\in [0,T].
\end{equation}

The estimate of the flux $f_{\eps,\delta }(u_{\eps,\delta } )  \nabla p_{\eps,\delta }$, similar to \cref{boundfluxL2} is not available in $L^2(\Omega_T)$,
but, for $s<d/2$, the following bound holds
 \begin{equation}\label{boundfluxq'}
	 \| f_{\eps,\delta }(u_{\eps,\delta } )  \nabla p_{\eps,\delta } \|^2_{L^2((0,T); L^{q'}(\Omega))} \leqslant   C,
\end{equation}
where
$$ q'= \frac{4d}{2d + n (d-2s)} < 2$$
and the constant $C$ is independent of $\eps$ and $\delta$.
We observe that $q'>1$ if and only if $n < \frac{2d}{d-2s}$.
In order to obtain \cref{boundfluxq'}, we use \cref{emb1},
taking into account that
\begin{equation}\label{ineqf}
	 f_{\eps,\delta}(z)\leqslant   C z^n \qquad\text{for all } z\in [0,+\infty).
\end{equation}
Using \cref{ineqf} and H\"older's inequality, we have
\begin{align*}
& \int_0^T\Big|\int\limits_{\Omega}  |f_{\eps,\delta}(u_{\eps,\delta})\nabla p_{\eps,\delta} |^{q'} \,\d x\Big|^{2/q'} \,\d t\\
&\leqslant    \int_0^T \Big(\int\limits_{\Omega}  |\sqrt{f_{\eps,\delta}(u_{\eps,\delta})}|^{2q'/(2-q')}\,\d x\Big)^{(2-q')/q'}
\int\limits_{\Omega}  \Big|\sqrt{f_{\eps,\delta}(u_{\eps,\delta})}\Big|^2|\nabla p_{\eps,\delta} |^{2} \,\d x \,\d t\\
&\leqslant   C \int_0^T \Big(\int\limits_{\Omega}  |u_{\eps,\delta}|^{nq'/(2-q')}\,\d x\Big)^{(2-q')/q'}
\int\limits_{\Omega}  {f_{\eps,\delta}(u_{\eps,\delta})}|\nabla p_{\eps,\delta} |^{2} \,\d x \,\d t.
\end{align*}
Since $nq'/(2-q')={2d}/(d-2s)$, by \cref{emb1} and \cref{boundHsepsdelta} and \cref{eq:energy_esti2} we obtain \cref{boundfluxq'}.

In the case $s\geqslant  d/2$, hence for $d=1$ because $s<1$, we obtain the simpler estimate for $q'=2$, using the embeddings \cref{emb2} and \cref{emb3}.

Using  \cref{boundfluxq'} and the formulation \cref{apr-002}, we obtain that there exists a constant $C$ independent of $\eps$ and $\delta$ such that
\begin{equation}\label{boundut}
	\|\partial_t  u_{\eps,\delta }\|_{L^2((0,T); (W^{1,q}(\Omega))^*)} \leqslant   C.
\end{equation}

Using \cref{boundHsepsdelta}, \cref{boundfluxq'}, and \cref{boundut}, by Aubin--Lions' compactness lemma (see \cite{simon87}), we obtain that there exists  $u\in C([0,T],L^{p}(\Omega))$, for $p < \frac{2d}{d - 2s}$,  such that 
\begin{equation}\label{strongCu}
u_{\eps,\delta}\xrightarrow{\eps,\delta\to 0}u \qquad   \text{ strongly in } C([0,T]; L^{p}(\Omega)), 
\end{equation}
and
\begin{equation}\label{strongLu}
u_{\eps,\delta}\xrightarrow{\eps,\delta\to 0}u \qquad   \text{ strongly in } L^2((0,T); H^{s}(\Omega)).
\end{equation}
Moreover, by \cref{boundHsplus1} we have that
\begin{equation}\label{Hsplus1conv}
 	u_{\eps,\delta}\xrightarrow{\eps,\delta\to 0}u\qquad \textrm{weakly in } L^2((0,T);H^{1+s}(\Omega)).
\end{equation}

By the interpolation inequality \cref{interpsemi}, from \cref{strongLu} and \cref{Hsplus1conv} we obtain
\begin{equation}\label{Hstroconv}
 	u_{\eps,\delta}\xrightarrow{\eps,\delta\to 0}u\qquad \textrm{strongly in } L^2((0,T);H^{1+r}(\Omega)) \qquad \text{for all } r \in [0,s),
\end{equation}
and we can also obtain that
\begin{equation}\label{a.e.conv}
	 u_{\eps,\delta}\xrightarrow{\eps,\delta\to 0}u, \qquad  \nabla u_{\eps,\delta}\xrightarrow{\eps,\delta\to 0}\nabla u \quad\text{a.\,e.~in }\Omega_T.
\end{equation}

Then
\begin{align*}
	&u \in L^{\infty}((0,T); H^s(\Omega)) \cap L^{2}((0,T); H^{s+1}(\Omega)) \cap C([0,T]; L^{q'}(\Omega)),\\
  & u(\cdot,t)  \geqslant  0 \quad \text{ a.\,e.~in } \bar{\Omega}, \ \text{for all}\quad t\in [0,T],
\end{align*}
and
\begin{equation}
\label{mass_cons3}
\int_{\Omega}u(x,t)\,\d x = \int_{\Omega} u_{0}(x) \,\d x \qquad\text{for all } t\in (0,T].
\end{equation}
Taking into account that
$$\norm{u_{\eps,\delta}}^2_{\dot H^{ s+1} (\Omega)}= \norm{p_{\eps,\delta}}^2_{\dot H^{ 1-s} (\Omega)},$$
we obtain also that

\begin{align}
\label{pplus1conv}
p_{\eps,\delta}&\xrightarrow{\eps,\delta\to 0}p\qquad \textrm{ weakly in } L^2((0,T);H^{1-s}(\Omega)), \\ \label{p1conv}	p_{\eps,\delta}&\xrightarrow{\eps,\delta\to 0}p\qquad \textrm{strongly in } L^2((0,T);H^{1-r}(\Omega)) \quad \text{for all } r\in (s,1],
\end{align}
and $p=(-\Delta)^su$.

\subsubsection{Limits for \texorpdfstring{$\eps, \delta\to 0$}{epsilon and delta tending to zero} in the weak formulation}

In order to pass to the limit in the weak formulation \cref{apr-002} of the equation, we fix
$v\in C^{\infty}_c(\overline{\Omega}\times [0,T))$ such that $\nabla v\cdot {\bf n} =0 $ on $(0,T)\times \partial\Omega$.
Taking into account that $u_{\eps,\delta}=0$ if and only if $ f_{\eps,\delta }(u_{\eps,\delta})=0$, 
we rewrite the term
\begin{equation}\label{nlpart}
\begin{aligned}
	-\iint_{\Omega_T}{ J_{\eps,\delta}(u_{\eps,\delta })    \cdot\nabla v  \,\d x\, \d t}
	&= \iint_{\Omega_T}p_{\eps,\delta}\nabla\left(f_{\eps,\delta}(u_{\eps,\delta})\right)\cdot\nabla v\,\d x\,\d t
\\ & \quad + \iint_{\Omega_T}f_{\eps,\delta}(u_{\eps,\delta})p_{\eps,\delta}\Delta v\,\d x\,\d t,
 \end{aligned}
\end{equation}
and the term
\begin{equation}\label{lpart}
	\int_{0}^T { \langle \partial_t u_{\eps,\delta } , v \rangle_{H^1(\Omega)^*, H^1(\Omega)} \,\d x}
	= -\iint_{\Omega_T} u_{\eps,\delta}\partial_t v\,\d x\,\d t -\int_{\Omega}u_{0,\eps,\delta}(x)v(x,0)\,\d x.
\end{equation}

We fix $r\in(s,1)$. Using \cref{p1conv} and the embedding \cref{emb1} we obtain that
\begin{equation}\label{p1conv2}
 	p_{\eps,\delta}\xrightarrow{\eps,\delta\to 0}p\qquad \textrm{strongly in } L^2((0,T);L^{q'}(\Omega)),
\end{equation}
where $q'=\frac{2d}{d+2(r-1)}$ and satisfies $q'>2$.
The conjugate exponent of $q'$ is $q=\frac{2d}{d-2(r-1)}$.
We observe that  $q>1$ if and only if $r>1-d/2$.
Owing to H\"older's inequality, 
\begin{equation}
    \left(\int_\Omega |u_{\eps,\delta}^{n-1}\nabla u_{\eps,\delta}|^q\,\d x\right)^{1/q}\leqslant 
    \left(\int_\Omega |u_{\eps,\delta}^{n-1}|^{q_1}\,\d x\right)^{1/{q_1}}\left(\int_\Omega |\nabla u_{\eps,\delta}|^{q_2}\,\d x\right)^{1/{q_2}} ,
\end{equation}
where $\frac{1}{q_1}+ \frac{1}{q_2}= \frac{1}{q}$ and $q_2=\frac{2d}{d-2s}$.
It follows that $q_1=\frac{d}{1+s-r}$ and $(n-1)q_1\leqslant  \frac{2d}{d-2s}$
if $n<\frac{d+2(1-s)}{d-2s}$. 

Using the embedding \cref{emb1} and the bound in \cref{boundHsepsdelta}, we deduce
\begin{equation}
    \left(\int_\Omega |u_{\eps,\delta}^{n-1}\nabla u_{\eps,\delta}|^q\,\d x\right)^{1/q}\leqslant 
    C\left(\int_\Omega |\nabla u_{\eps,\delta}|^{q_2}\,\d x\right)^{1/{q_2}}.
\end{equation}
Since $f'_{\eps,\delta}(z)\leqslant  C z^{n-1}$ for any $z\geqslant 0$, from the embedding \cref{emb1} and \cref{boundHsplus1} we obtain that
\begin{equation}\label{l2qb}
    \int_0^T\left(\int_\Omega |f_{\eps,\delta}'(u_{\eps,\delta})\nabla u_{\eps,\delta}|^q\,\d x\right)^{2/q}\,\d t\leqslant     C.
\end{equation}
The last bound and the convergence \cref{a.e.conv} imply that
\begin{equation}\label{u1conv2}
 	f_{\eps,\delta}'(u_{\eps,\delta})\nabla u_{\eps,\delta}\xrightarrow{\eps,\delta\to 0}nu^{n-1}\nabla u \qquad \textrm{weakly in } L^2((0,T);L^{q}(\Omega)).
\end{equation}
Moreover, since $f_{\eps,\delta}(z)\leqslant  C z^{n}$ for any $z\geqslant 0$,
and $nq\leqslant  \frac{2d}{d-2s}$ if $n<\frac{d+2(1-s)}{d-2s}$,
using the embedding \cref{emb1} and the bound \cref{boundHsepsdelta},
we obtain that
\begin{equation}\label{eq:bl2q}
    \int_0^T\left(\int_\Omega |f_{\eps,\delta}(u_{\eps,\delta})|^q\,\d x\right)^{2/q}\,\d t\leqslant 
    C.
\end{equation}
The last bound and the convergence \cref{a.e.conv} imply that
\begin{equation}\label{u1conv3}
 	f_{\eps,\delta}(u_{\eps,\delta})\xrightarrow{\eps,\delta\to 0}u^{n} \qquad \textrm{weakly in } L^2((0,T);L^{q}(\Omega)).
\end{equation}
Using \cref{u1conv2}, \cref{u1conv3} and  \cref{p1conv2}, we can pass to the limit in \cref{nlpart} obtaining that
\begin{equation}\label{nlpartc}
\begin{aligned}
	&\lim_{\eps,\delta\to 0} \Bigl( \iint_{\Omega_T}p_{\eps,\delta}\nabla\left(f_{\eps,\delta}(u_{\eps,\delta})\right)\cdot\nabla v\,\d x\,\d t
 + \iint_{\Omega_T}f_{\eps,\delta}(u_{\eps,\delta})p_{\eps,\delta}\Delta v\,\d x\,\d t \Bigr)
 \\ &\quad = n\iint_{\Omega_T}pu^{n-1}\nabla u \cdot\nabla v\,\d x\,\d t
 + \iint_{\Omega_T}u^np\Delta v\,\d x\,\d t.
 \end{aligned}
\end{equation}
Using \cref{strongCu}, we  pass to the limit in \cref{lpart} obtaining that
\begin{equation}\label{lpartc}
	 \lim_{\eps,\delta\to 0}  \Bigl( \iint_{\Omega_T}u_{\eps,\delta}\partial_t v\,\d x\,\d t +\int_{\Omega}u_{0,\eps,\delta}(x)v(x,0)\,\d x \Bigr)
	= \iint_{\Omega_T} u\partial_t v\,\d x\,\d t +\int_{\Omega}u_{0}(x)v(x,0)\,\d x.
\end{equation}
By \cref{nlpartc} and \cref{lpartc}, we pass to the limit in the formulation \cref{apr-002} and we prove \cref{eq:weak_sol}.

\subsubsection{Limits for \texorpdfstring{$\eps, \delta\to 0$}{epsilon and delta tending to zero} in the entropy and energy estimates}

In order to prove the entropy inequality \cref{eq:entropy},
we fix $t\in (0,T]$ and,
by \cref{strongCu}, we can assume that
\begin{equation}\label{t.a.e.conv}
	 u_{\eps,\delta}(\cdot,t)\xrightarrow{\eps,\delta\to 0}u(\cdot,t) \quad\text{a.\,e.~in }\Omega.
\end{equation}
Starting from \cref{diff_G}, we obtain for any $z>0$ that
\begin{equation*}
	G_{\eps,\delta}(z)=G_0(z)+\frac{\eps}{(\alpha-1)(\alpha-2)}{z^{2-\alpha}} +\frac{\eps}{\alpha-1}z - \frac{\eps}{\alpha-2}
	+ \frac{\delta}{2}z^2 + \frac{\delta}{2} -\delta z,
\end{equation*}
and then
\begin{equation}\label{diff_G2}
	G_{\eps,\delta}(z)\geqslant G_0(z) - \frac{\eps}{\alpha-2}
	+ \frac{\delta}{2} -\delta z.
\end{equation}

Using \cref{diff_G2}, \cref{t.a.e.conv} and Fatou's lemma we obtain
\begin{equation}\label{lscGeps}
    \liminf_{\eps,\delta\to 0}\int_{\Omega}G_{\eps,\delta}(u_{\eps,\delta}(x,t))\,\d x\geqslant  \int_{\Omega}G_{0}(u(x,t))\,\d x.
\end{equation}
	
Since \cref{Hsplus1conv}, \cref{lscGeps} and \cref{convid} hold,
passing to the limit as $\eps,\delta\to 0$ in \cref{eq:entropy_approx2} we obtain
\begin{equation*}
    \int_\Omega G_0(u(x,t)) \,\d x + \int_0^t{\| u(\cdot,r)\|_{\overset{.}{H}^{s+1}(\Omega)}^2 \,\d r \leqslant   \int_\Omega G_0(u_0) \,\d x}.
\end{equation*}

We prove \cref{eq:ineq}. 
From \cref{eq:energy_esti2} for $t=T$ it follows that
\begin{equation}
\iint_{\Omega_T}\abs{f^{1/2}_{\eps,\delta}(u_{\eps,\delta})\nabla p_{\eps,\delta}}^2 \,\d x \,\d t \leqslant  C.
\end{equation}
Then there exists $g\in L^2(\Omega_T;\R^d)$ such that
\begin{equation*}
 	f^{1/2}_{\eps,\delta}(u_{\eps,\delta})\nabla p_{\eps,\delta}\xrightarrow{\eps,\delta\to 0} g \qquad \textrm{weakly in } L^2((0,T);L^{2}(\Omega;\R^d)).
\end{equation*}
By a lower-semicontinuity argument, we can pass to the limit in \cref{eq:energy_esti2}
obtaining \eqref{eq:ineq}. 

In order to identify $g$, let us fix $\phi \in C_c^\infty ({\Omega}_T;\R^d)$. 
By \cref{boundfluxq'}, there exists  $h\in L^2((0,T);L^{q'}(\Omega;\R^d))$ such that
\begin{equation}\label{u1conv5}
 	f^{1/2}_{\eps,\delta}(u_{\eps,\delta})f^{1/2}_{\eps,\delta}(u_{\eps,\delta})\nabla p_{\eps,\delta}\xrightarrow{\eps,\delta\to 0} h \qquad \textrm{weakly in } L^2((0,T);L^{q'}(\Omega;\R^d)).
\end{equation}
Therefore it follows that $h=u^{n/2}g$.
By \cref{nlpartc}, we have 
\begin{equation}\label{nlpartd}
\begin{aligned}
	-\iint_{\Omega_T}u^{n/2}g\cdot\phi\,\d x\,\d t 
 &= n\iint_{\Omega_T}pu^{n-1}\nabla u \cdot\phi\,\d x\,\d t+ \iint_{\Omega_T}u^np\div\phi\,\d x\,\d t.
 \end{aligned}
\end{equation}
Since $u^{n/2}\in L^2(\Omega_T)$ and $pu^{n-1}\nabla u\in L^1(\Omega_T)$, it follows that $\nabla\left(u^np\right)\in L^1(\Omega_T)$ and
$$\nabla\left(u^np\right)= u^{n/2}g + npu^{n-1}\nabla u.$$
Then \cref{eq:pseudo_flux_weak} holds.

Let us assume that $n\geqslant  2$. By \cref{eq:entropy} we have that $\int_\Omega G_0(u(x,t)) \,\d x <+\infty$ for any $t\in (0,T)$.
By the definition of $G_0$, we obtain \cref{positivityu}.

By \cref{eq:f'}, \cref{eq:f0} we observe that $f^{1/2}_{\eps,\delta}$ is Lipschitz continuous and that $ f'_{\eps,\delta }(0)=0$. 
As before,  let us fix $\phi \in C_c^\infty ({\Omega}_T;\R^d)$. 
Taking into account \cref{positivityu}, we can write
$$ \iint_{\Omega_T} \nabla(f^{1/2}_{\eps,\delta}(u_{\eps,\delta}))\cdot\phi \,\d x \,\d t = 
 \frac{1}{2}\iint_{\Omega_T}f^{-1/2}_{\eps,\delta}(u_{\eps,\delta})f^{'}_{\eps,\delta}(u_{\eps,\delta})\nabla u_{\eps,\delta}\cdot\phi \,\d x \,\d t .$$
Then we write
\begin{align*} \iint_{\Omega_T}f^{1/2}_{\eps,\delta}(u_{\eps,\delta})\nabla p_{\eps,\delta}\cdot\phi \,\d x \,\d t &=
-\iint_{\Omega_T} {f^{1/2}_{\eps,\delta}(u_{\eps,\delta})p_{\eps,\delta} \div\phi \,\d x \, \d t}\\ &\quad - \frac{1}{2}\iint_{\Omega_T} { f^{-1/2}_{\eps,\delta}(u_{\eps,\delta})f^{'}_{\eps,\delta}(u_{\eps,\delta})p_{\eps,\delta} \nabla u_{\eps,\delta} \cdot \phi \,\d x \,\d t}.
\end{align*}
We have to pass to the limit in the right-hand side of the last equality.
By \cref{eq:bl2q} and \cref{a.e.conv} we obtain that 
\begin{equation}
 	f^{1/2}_{\eps,\delta}(u_{\eps,\delta})\xrightarrow{\eps,\delta\to 0}u^{\frac{n}{2}} \qquad \textrm{weakly in } L^2((0,T);L^{q}(\Omega)),
\end{equation}
(and, owing to \cref{strongCu}, strongly in $C((0,T); L^p(\Omega)$ as well). By \cref{p1conv2}, we have 
\begin{equation}
    \lim_{\eps,\delta\to 0} \iint_{\Omega_T} {f^{1/2}_{\eps,\delta}(u_{\eps,\delta})p_{\eps,\delta} \div\phi \,\d x \, \d t}
    = \iint_{\Omega_T}{u^{\frac{n}{2}}p\div\phi\,\d x\,\d t}.
\end{equation}
In the same way of the proof of \cref{l2qb}, we prove that
\begin{equation}\label{l2qb'}
    \int_0^T\left(\int_\Omega |f^{-1/2}_{\eps,\delta}(u_{\eps,\delta})f_{\eps,\delta}'(u_{\eps,\delta})\nabla u_{\eps,\delta}|^q\,\d x\right)^{2/q}\,\d t\leqslant     C.
\end{equation}
Using the pointwise limit and the fact that $n\geqslant  2$, we prove that
\begin{equation}\label{u1conv4}
 	 f^{-1/2}_{\eps,\delta}(u_{\eps,\delta})f_{\eps,\delta}'(u_{\eps,\delta})\nabla u_{\eps,\delta}\xrightarrow{\eps,\delta\to 0}nu^{\frac{n}{2}-1}\nabla u \qquad \textrm{weakly in } L^2((0,T);L^{q}(\Omega)).
\end{equation}
By \cref{u1conv4} and  \cref{p1conv2}, we obtain
\begin{equation}
    \lim_{\eps,\delta\to 0} \frac{1}{2}\iint_{\Omega_T} { f^{-1/2}_{\eps,\delta}(u_{\eps,\delta})f^{'}_{\eps,\delta}(u_{\eps,\delta})p_{\eps,\delta} \nabla u_{\eps,\delta} \cdot \phi \,\d x\,\d t}
    = \frac{n}{2}\iint_{\Omega_T}{u^{\frac{n}{2}-1}p\nabla u \cdot\phi\,\d x\,\d t}.
\end{equation}
Then,
 $$ \iint_{\Omega_T}g\cdot\phi \,\d x \,\d t = -\iint_{\Omega_T} { u^{\frac{n}{2}}p \div\phi \,\d x \, \d t}
- \frac{n}{2}\iint_{\Omega_T} { u^{\frac{n}{2}-1} p \nabla u \cdot \phi \,\d x \,\d t}. $$

\subsubsection{Strict positivity}

Let us assume that  $n> 2 + \frac{2d}{2(s+1)-d}$.

By \cref{eq:entropy}, we have  
 \[
 	\|u(\cdot,t)\|_{H^{s+1}({\Omega})}<+\infty  \qquad \textrm{ for a.\,e.~} t \in (0,T).
 \]
Using the embedding \cref{emb4}, we deduce 
 \[
 	\|u(\cdot,t)\|_{C^{s+1 - d/2}(\overline{\Omega})}<+\infty  \qquad \textrm{ for a.\,e.~} t \in (0,T).
 \]
The last bound and the condition on $n$ imply
\cref{strictpositivity}.
Indeed, assuming that \cref{strictpositivity} does not hold, we fix $t\in(0,T)$ such that $\|u(\cdot,t)\|_{C^{s+1 - d/2}(\overline{\Omega})}<+\infty$, 
and $x_0 \in \overline\Omega$ such that  $u(x_0,t)=0$.
Then there exists a constant $C$ such that
\begin{equation}\label{holdercont}
	0\leqslant   u(x ,t) \leqslant   C |x - x_0|^{s+1 - d/2} \qquad\text{for all }x\in \overline\Omega.
\end{equation}
By the definition of $G_0$, we have 
$$
	G_{0}(u) \geqslant   \frac{1}{(n-2)(n-1)}\frac{1}{u^{n-2}}.
$$
Using the previous inequality and \cref{holdercont}, we obtain
\begin{equation}\label{disent}
	 \int_{\Omega} {G_0(u(x,t)) \,\d x}
 	\geqslant  \frac{C^{2-n}}{(n-1)(n-2)} \int_{\Omega} {|x - x_0|^{(2-n)(s+1-d/2)}\,\d x}.
\end{equation}
Since $n>2+\frac{2d}{2(s+1) - d}$, then
\[\int_{\Omega} |x - x_0|^{(2-n)\left(s+1-\frac d2\right)}\,\d x=+\infty \qquad \text{and, by \cref{disent},}\qquad \displaystyle\int_{\Omega} G_0(u(x,t)) \,\d x=+\infty.\]
Since \cref{eq:entropy} implies $\displaystyle\int_{\Omega} G_{0}(u(x,t)) \,\d x<+\infty$, we have a contradiction.

\section{Local entropy estimate}
\label{LEE}

In this section, we prove local entropy estimates for the weak solutions constructed in \cref{Th-ex}. These estimates will play a fundamental role in the proof of the finite speed of propagation and of the waiting time phenomenon.

In order to avoid technical problems and to highlight the main difficulties, we concentrate on a very simple geometry and restrict ourselves to the case in which $\Omega = B_R(0)$, with $R>0$.

For $S\in (0,R)$, we define
\begin{equation}
\label{eq:OmegaS}
\Omega(S) \coloneqq \left\{x\in \R^d: S<\abs{x}<R\right\}, \qquad \textrm{ and } \qquad \Omega_T(S) \coloneqq \Omega(S)\times (0,T).
\end{equation}
Moreover, for $S\in (0,R)$ and $\sigma \in (0,R-S)$, we consider the cut-off functions $\psi_{S,\sigma}\in C^\infty(\overline\Omega)$ such that
\begin{equation}\label{e-9}
\psi_{S,\sigma } (x) =
\begin{cases}
1 & \text{ if } x \in \Omega(S +  \sigma),\\
0 & \text{ if } x \in  B_S(0),
\end{cases}
\qquad\qquad 0<\psi_{S,\sigma }(x)<1\quad\text { if } S<|x|<S+\sigma,
\end{equation}
and
\begin{equation}\label{psiest}
	|\nabla \psi_{S,\sigma }| \leqslant   \frac{C} {\sigma },\quad | (-\Delta)^{ s +1 } \psi_{S,\sigma } | \leqslant   \frac{C} {  \sigma^{2(s+1)} },
\end{equation}
for a constant $C$ depending  on $d$ only.

\begin{lemma}\label{lemma-lee}

Let $u$ be a solution of problem \cref{eq:ft} given by  \cref{Th-ex}.
If $s\in(\frac{(d-2)_+}{2},1)$ and $n\in (1, \frac{s+2}{s+1})$,
then there exists a constant $C$ such that,
\begin{equation}\label{rt-2}
\begin{aligned}
&\int_{\Omega(S+  \sigma)}{G_0(u(x,T)) \,\d x}  +
\frac{1}{2} \iint_{\Omega_T(S +\sigma) }{ |(-\Delta)^{\frac{s+1}{2}}(u \,\psi_{S,\sigma } ) |^2   \,\d x \, \d t}  \\&\leqslant   \int_{\Omega(S)}{G_0(u_0) \,\d x} +
\frac{C}{\sigma^{2(s+1)}}  \int_{0}^T { \| u \|^{2}_{L^{2}(\Omega(S))} \,\d t} +
\frac{C }{\sigma^{2(s+1)}}  \left( \int_{0}^T { \| u \|^{2}_{L^{2}(\Omega(S))} \,\d t} \right)^{\varpi},
\end{aligned}
\end{equation}
where $\varpi \coloneqq  \min \{ \frac{s}{2s+1}, 1 -(n-1)(s+1), \frac{2ns - d(n-1)}{4s} \} < 1 $.
\end{lemma}

\begin{proof}[Proof of \cref{lemma-lee}]

Let us consider the approximated smooth classical solutions $u_{\varepsilon,\delta,\gamma}^N$
and the regularized entropy $G_{\varepsilon,\delta,\gamma}$ defined in \cref{defGepsdeltagamma}.

Next, for simplicity, we denote by $u \coloneqq  u^N_{\varepsilon,\delta,\gamma}$ in $\Omega_T$ and $\psi\coloneqq \psi_{S,\sigma}$.
Since $u$ are classical solutions of problem \cref{eq:ft-r},  multiplying equation in \cref{eq:ft-r} by $G'_{\varepsilon,\delta,\gamma}(u(\cdot,t)) \psi^3(\cdot)$
and integrating in $\Omega$ we have
\begin{align}
\label{eq:conto1}
\frac{\d}{\d t}\int_\Omega G_{\varepsilon,\delta}(u) \psi^3 \,\d x
&= \int_\Omega \psi^3 G'_{\varepsilon,\delta,\gamma}(u) \div( f_{\varepsilon,\delta,\gamma}(u) \nabla p) \,\d x \nonumber\\
&= -\int_\Omega \psi^3 f_{\varepsilon,\delta,\gamma}(u) G''_{\varepsilon,\delta\gamma}(u)\nabla u\cdot \nabla p\,\d x
- \int_\Omega f_{\varepsilon,\delta,\gamma}(u) G'_{\varepsilon,\delta,\gamma}(u) \nabla \psi^3 \cdot \nabla p\,\d x \nonumber\\
& = -\int_\Omega \psi^3 \nabla u\cdot \nabla p\,\d x + \int p\div\left(f_{\varepsilon,\delta,\gamma}(u) G'_{\varepsilon,\delta,\gamma}(u)  \nabla \psi^3 \right)\,\d x\nonumber\\
& = \int_\Omega u \psi^3 \Delta p\,\d x + \int u \nabla \psi^3 \cdot \nabla p \,\d x \nonumber\\
&\quad + \int_\Omega p\left(f_{\varepsilon,\delta,\gamma}(u) G'_{\varepsilon,\delta,\gamma}(u) \right)'\nabla u\cdot \nabla \psi^3 \,\d x +
\int_\Omega p f_{\varepsilon,\delta,\gamma}(u) G'_{\varepsilon,\delta,\gamma}(u)  \Delta \psi^3 \,\d x \nonumber\\
& \eqqcolon  A + B + C +D.
\end{align}

We develop the terms of the last expression.
Given $\beta>0$ and $w,v\in H^\beta_N(\Omega)$ we define
$$
R_{\beta}(w,v) = (-\Delta)^{\beta}(w\,v) - w (-\Delta)^{\beta} v - v (-\Delta)^{\beta}w.
$$
Using \cref{lem:ip} and the previous notation we obtain
\begin{align}
\label{eq:A}
A &=   - \int_\Omega u \psi^2 \psi ( -\Delta)^{s+1} u \,\d x   \nonumber\\
& = - \int_\Omega u \psi^2   ( -\Delta)^{s+1} (u \psi) \,\d x + \int_\Omega u^2 \psi^2   ( -\Delta)^{s+1} \psi \,\d x
+ \int_\Omega u \psi^2 R_{s+1} (u,\psi) \,\d x   \nonumber\\
& = -\int_\Omega \Dusm\left(u\psi^2\right)\Dusm (u \psi) \,\d x + \int_\Omega u^2 \psi^2   ( -\Delta)^{s+1} \psi \,\d x
+ \int_\Omega u \psi^2 R_{s+1} (u,\psi) \,\d x   \nonumber\\
& = -\int_\Omega \abs{\Dusm (u\psi) }^2 \psi \,\d x  - \int_\Omega   R_{\frac{1+s}{2}}(u\psi,\psi) \Dusm  (u \psi)  \,\d x \nonumber\\
& \quad -  \int_\Omega  u \psi  \Dusm\left(u\psi \right)\Dusm  \psi \, \d x + \int_\Omega u^2 \psi^2   ( -\Delta)^{s+1} \psi \,\d x \nonumber\\
&\quad  + \int_\Omega u \psi^2 R_{s+1} (u,\psi) \,\d x .
\end{align}
\begin{align}\label{eq:A-0}
  B  &  = -\int_\Omega p\div(u\nabla \psi^3) \,\d x =  -\int_\Omega p\nabla u\cdot \nabla \psi^3 \,\d x -\int_\Omega pu \Delta \psi^3 \,\d x  \nonumber\\
     & =  - 3 \int_\Omega p \psi \nabla (u \psi) \cdot\nabla \psi  \,\d x - 3 \int_\Omega p \psi u ( |\nabla \psi |^2 + \psi \Delta \psi)  \,\d x.
\end{align}

On the other hand, since for any $z>0$, 
$$
(f_{\varepsilon,\delta,\gamma}(z) G'_{\varepsilon,\delta,\gamma}(z))' = 1 +  f'_{\varepsilon,\delta,\gamma}(z) G'_{\varepsilon,\delta,\gamma}(z) , \qquad
f_{\varepsilon,\delta,\gamma}(z) G'_{\varepsilon,\delta,\gamma}(z) = z    + \int_0^z{ f'_{\varepsilon,\delta.\gamma}(r) G'_{\varepsilon,\delta,\gamma}(r) \, \d r},
$$
then
\begin{align}
\label{eq:B}
\begin{aligned}
C &= \int_\Omega p\left(f_{\varepsilon,\delta,\gamma}(u) G'_{\varepsilon,\delta,\gamma}(u) \right)'\nabla u\cdot \nabla \psi^3 \,\d x
\\ &=  \int_\Omega p\nabla u\cdot\nabla \psi^3 \,\d x  +  \int_\Omega p f'_{\varepsilon,\delta,\gamma}(u) G'_{\varepsilon,\delta,\gamma}(u) \nabla u\cdot \nabla \psi^3 \,\d x,
\end{aligned}
\end{align}
and
\begin{align}
\label{eq:C}
D & = \int_\Omega p f_{\varepsilon,\delta}(u) G'_{\varepsilon,\delta}(u) \Delta \psi^3 \,\d x
=  \int_\Omega p u \Delta \psi^3 \,\d x +
\int_\Omega p \Bigl(   \int_0^u{ f'_{\varepsilon,\delta}(r) G'_{\varepsilon,\delta}(r) \, \d r}   \Bigr) \Delta \psi^3 \d x.
\end{align}
Taking into account that
$ \nabla u\cdot \nabla \psi^3 = 3\psi\nabla (u\psi) -3\psi|\nabla\psi|^2 $ and $ \Delta \psi^3 = 3(\psi^2\Delta\psi +2\psi|\nabla\psi|^2)$
we have that
\begin{align*}
B+C+D &= - 3  \int_\Omega p \psi u |\nabla\psi|^2  \,\d x + 3  \int_\Omega p u \psi^2  \Delta\psi \,\d x \\
& \qquad +3  \int_\Omega p \psi  f'_{\varepsilon,\delta,\gamma}(u) G'_{\varepsilon,\delta,\gamma}(u)  \nabla (u \psi)  \cdot\nabla \psi \,\d x -
3 \int_\Omega p \psi u f'_{\varepsilon,\delta,\gamma}(u) G'_{\varepsilon,\delta,\gamma}(u)    |\nabla \psi |^2    \,\d x  \nonumber\\
& \qquad + 3 \int_\Omega p \psi \Bigl(   \int_0^u{ f'_{\varepsilon,\delta,\gamma}(v) G'_{\varepsilon,\delta,\gamma}(v) \, \d v}   \Bigr) \bigl(2 |\nabla \psi |^2 +  \psi \Delta \psi \bigr) \, \d x .
\end{align*}
Summing up, we obtain the following relation
\begin{align}
\label{eq:weighted_entropy1}
&\frac{\d}{\d t}\int_\Omega G_{\varepsilon,\delta} (u) \psi^3 \,\d x   
 +\int_\Omega \abs{\Dusm (u\psi) }^2 \psi \, \d x \nonumber\\
& =  \int_\Omega u^2 \psi^2   ( -\Delta)^{s+1} \psi \,\d x + \int_\Omega u \psi^2 R_{s+1} (u,\psi) \,\d x  \nonumber\\
&\qquad  - \int_\Omega   R_{\frac{1+s}{2}}(u\psi,\psi) \Dusm  (u \psi) \, \d x -  \int_\Omega  u \psi  \Dusm\left(u\psi \right)\Dusm  \psi \, \d x  \nonumber\\
&\qquad - 3  \int_\Omega p \psi u |\nabla\psi|^2  \,\d x + 3  \int_\Omega p u \psi^2  \Delta\psi \,\d x \\
&\qquad +3  \int_\Omega p \psi  f'_{\varepsilon,\delta,\gamma}(u) G'_{\varepsilon,\delta,\gamma}(u)  \nabla (u \psi)  \cdot\nabla \psi \, \d x
 \nonumber\\ &\qquad -
3 \int_\Omega p \psi u f'_{\varepsilon,\delta,\gamma}(u) G'_{\varepsilon,\delta,\gamma}(u)    |\nabla \psi |^2 \,\d x  \nonumber\\
&\qquad + 3 \int_\Omega p \psi \Bigl(\int_0^u{ f'_{\varepsilon,\delta,\gamma}(v) G'_{\varepsilon,\delta,\gamma}(v) \, \d v}   \Bigr) \bigl(2 |\nabla \psi |^2 +  \psi \Delta \psi \bigr)  \,\d x .\nonumber
\end{align}

In order to estimate the right-hand side of \cref{eq:weighted_entropy1}, we observe that there exists constants $c_1, c_2 >0$ depending only on the dimension $d$ such that,
for any $\gamma>0$ and $\eps,\delta\in (0,1)$, 
\begin{equation}\label{fGest}
|f'_{\varepsilon,\delta,\gamma}(z) G'_{\varepsilon,\delta,\gamma}(z)|   \leqslant   c_1 z^{n-1} + c_2 \qquad\text{for all }z\in (0,+\infty).
\end{equation}
Indeed, since $|G'_{\varepsilon,\delta,\gamma}(z)|\leqslant   |G'_{\varepsilon,\delta}(z)|$, $f'_{\varepsilon,\delta,\gamma}=f'_{\varepsilon,\delta}$ for any $z\in (0,+\infty)$ and
\begin{align*}
|f'_{\varepsilon,\delta}(z) G'_{\varepsilon,\delta}(z)| &\leqslant  
\frac{z^{n+\alpha -1}(n z^{\alpha} + \varepsilon \alpha z^n)}{(z^{\alpha} + \varepsilon z^n + \delta z^{n+\alpha})^2}
\left| G'_0(z) - \frac{\varepsilon}{\alpha -1} z^{1-\alpha} + \delta \, z +  \frac{\varepsilon}{\alpha -1} - \delta \right |
\\ & \leqslant  
 \frac{4n + \alpha}{4}  z^{n-1}|G'_0(z)| +  \frac{n + 4\alpha}{4(\alpha -1)} + \frac{n +\alpha}{4} +
 \frac{4n + \alpha}{4}\left(\frac{\varepsilon}{\alpha -1} + \delta\right)  z^{n-1} ,
\end{align*}
we obtain \cref{fGest}.

Next, we will use the following inequalities:\\
 for $\beta \in (0,2)$, and $w,v\in H^\beta_N(\Omega)$,
\begin{equation}\label{leibest}
	\| R_{\beta}(w,v) \|_{L^2(\Omega)} \leqslant   \|w\|_{L^2(\Omega)} \| (-\Delta)^{\beta}v\|_{L^{\infty}(\Omega)},
\end{equation}
(see \cref{lm:leibnitz});\\
for $\beta \in (0, \frac{1+s}{2})$
\begin{equation}\label{eq:interpbeta}
\| \psi (- \Delta)^{\beta} u \|_{L^2(\Omega)} \leqslant   \| \Dusm (u \psi)  \|^{\theta}_{L^2(\Omega) }
\|  u \psi   \|^{1 - \theta}_{L^2(\Omega) }  + \frac{C}{\sigma^{2\beta}}\|u \|_{L^2(\Omega) } ,
\end{equation}
where $ \theta = \frac{2\beta}{1+s}$ (see \cref{lem-fr-3}).

By \cref{psiest} the first term of the right-hand side of \cref{eq:weighted_entropy1} can be estimated as
$$
	\int_\Omega u^2 \psi^2( -\Delta)^{s+1} \psi \,\d x \leqslant   \frac{C}{\sigma^{2(s+1)}}  \| u \psi \|^2_{L^2(\Omega )},
$$
whereas the second term, using \cref{leibest} and \cref{psiest}, can be estimated as
$$
\int_\Omega u \psi^2 R_{s+1} (u,\psi) \,\d x  \leqslant   \| u \psi \|_{L^2(\Omega)} \| \psi R_{s+1} (u,\psi) \|_{L^2(\Omega)}
\leqslant   \frac{C}{\sigma^{2(s+1)}}  \| u \psi \|_{L^2(\Omega )} \| u \|_{L^2(\Omega)}.
$$
For the third term, using \cref{leibest} and \cref{psiest}, for any $\epsilon >0$ we have that
\begin{align*}
\int_\Omega   R_{\frac{1+s}{2}}(u\psi,\psi) \Dusm  (u \psi)  \,\d x
&\leqslant    \| R_{\frac{1+s}{2}}(u\psi,\psi) \|_{L^2(\Omega )} \| \Dusm  (u \psi) \|_{L^2(\Omega )}
 \\ &\leqslant  
 \frac{C}{\sigma^{ s+1 }} \|  u \psi  \|_{L^2(\Omega )}  \| \Dusm  (u \psi) \|_{L^2(\Omega )} \\
&\leqslant  
\epsilon \|\Dusm  (u \psi) \|^2_{L^2(\Omega )} + \frac{C}{4\epsilon\sigma^{2(s+1)}} \|  u \psi  \|^2_{L^2(\Omega )}.
\end{align*}
For the forth term, using \cref{psiest}, for any $\epsilon >0$ we have that
\begin{align*}
-\int_\Omega  u \psi  \Dusm\left(u\psi \right)\Dusm \psi \,\d x  &\leqslant   \frac{C}{\sigma^{s+1}} \| \Dusm  (u \psi) \|_{L^2(\Omega )}  \|  u \psi  \|_{L^2(\Omega )} \\
&\leqslant   \epsilon \| \Dusm  (u \psi) \|^2_{L^2(\Omega )} + \frac{C}{4\epsilon\sigma^{2(s+1)}} \|  u \psi  \|^2_{L^2(\Omega )}.
\end{align*}
Still using  \cref{psiest} we have
\begin{equation*}
	- 3  \int_\Omega p \psi u |\nabla\psi|^2  \,\d x + 3  \int_\Omega p u \psi^2  \Delta\psi \,\d x \leqslant    \frac{C}{\sigma^{2 }}\| p \psi  \|_{L^2(\Omega )}  \|  u  \|_{L^2(\Omega(S))}.
\end{equation*}
Using \cref{fGest} and \cref{psiest} we obtain
\begin{align*}
 &3  \int_\Omega p \psi  f'_{\varepsilon,\delta,\gamma}(u) G'_{\varepsilon,\delta,\gamma}(u)  \nabla (u \psi)  \cdot\nabla \psi  \,\d x \\ & \leqslant  
  3  \int_\Omega |p \psi| (c_1u^{n-1}+c_2) |\nabla (u \psi)|\, |\nabla \psi|  \,\d x \\
  &\leqslant  
\frac{C}{\sigma }
 \| p \psi  \|_{L^2(\Omega )}  \| \nabla (u \psi) \|_{L^2(\Omega )}  \|  u  \|^{n-1}_{L^{\infty}(\Omega(S) )}  +  \frac{C}{\sigma }  \| p \psi  \|_{L^2(\Omega )}  \| \nabla (u \psi) \|_{L^2(\Omega )}.
\end{align*}
Analogously,
\begin{align*}
- 3 \int_\Omega p \psi u f'_{\varepsilon,\delta,\gamma}(u) G'_{\varepsilon,\delta,\gamma}(u)    \abs{\nabla \psi }^2    \d x   & \leqslant  
\frac{C}{\sigma^{2 }} \| p \psi  \|_{L^2(\Omega )}  \|  u  \|^n_{L^{2n}(\Omega(S))}  + \frac{C}{\sigma^{2 }}
\| p \psi  \|_{L^2(\Omega )}  \|  u  \|_{L^2(\Omega(S))}.
\end{align*}
Since from \cref{fGest} we have
$$\abs{\int_0^u{ f'_{\varepsilon,\delta,\gamma}(r) G'_{\varepsilon,\delta,\gamma}(r) \, \d r}} \leqslant   C(u^n+u),$$
then
\begin{align*}
& 3 \int_\Omega p \psi \left(\int_0^u{ f'_{\varepsilon,\delta,\gamma}(r) G'_{\varepsilon,\delta,\gamma}(r) \, \d r}\right) \left( 2 |\nabla \psi |^2 +  \psi \Delta \psi \right)  \d x \\
& \leqslant  
\frac{C}{\sigma^{2 }} \| p\psi\|_{L^2(\Omega )}  \|u\|^n_{L^{2n}(\Omega(S))}
 + \frac{C}{\sigma^{2}}\| p\psi  \|_{L^2(\Omega )}  \|  u  \|_{L^2(\Omega(S))}.
\end{align*}
Using the interpolation \cref{dong-01} for $\beta=1$ and \cref{eq:interpbeta} for $\beta=s$, we have
\begin{align*}
 \frac{C}{\sigma }
 \| p \psi  \|_{L^2(\Omega )}  \| \nabla (u \psi) \|_{L^2(\Omega )} &\leqslant  
\frac{C}{\sigma } \| p \psi  \|_{L^2(\Omega )} \| \Dusm  (u \psi) \|^{\frac{1}{1+s}}_{L^2(\Omega )} \|  u \psi  \|^{\frac{s}{1+s}}_{L^2(\Omega )} \\
& \leqslant   \frac{C}{\sigma } \| \Dusm  (u \psi) \|^{\frac{2s+ 1}{1+s}}_{L^2(\Omega )}
\|  u \psi  \|^{\frac{1}{1+s}}_{L^2(\Omega )}
\\ &\qquad +
\frac{C}{\sigma^{2s +1} } \| \Dusm  (u \psi) \|^{\frac{1}{1+s}}_{L^2(\Omega )}
\|  u \psi  \|^{\frac{s}{1+s}}_{L^2(\Omega )} \| u \|_{L^2(\Omega)} \\
&\leqslant   \epsilon \| \Dusm  (u \psi) \|^2_{L^2(\Omega )} +  \frac{C(\epsilon)}{\delta^{2(s+1)}} \|  u \psi  \|^2_{L^2(\Omega )} \\ & \qquad +
\frac{C(\epsilon)}{\sigma^{2(s+1)}} \|  u \psi  \|^{ \frac{2s}{2s+1}}_{L^2(\Omega )} \| u \|^{ \frac{2(1+s)}{2s+1}}_{L^2(\Omega )}.
\end{align*}
Analogously,
\begin{align*}
  & \frac{C}{\sigma }\| p \psi  \|_{L^2(\Omega )}  \| \nabla (u \psi) \|_{L^2(\Omega )}  \|  u  \|^{n-1}_{L^{\infty}(\Omega(S) )}\\  &\leqslant  
\epsilon \| \Dusm  (u \psi) \|^2_{L^2(\Omega )}
+\frac{C(\epsilon)}{\sigma^{2(s+1)}} \|  u \psi  \|^2_{L^2(\Omega )}
\|  u  \|^{2(n-1)(s+1)}_{L^{\infty}(\Omega(S) )} \\
&\quad +\frac{C(\epsilon)}{\sigma^{2(s+1)}} \|  u \psi  \|^{ \frac{2s}{2s+1}}_{L^2(\Omega )} \| u \|^{ \frac{2(1+s)}{2s+1}}_{L^2(\Omega )}
\|  u  \|^{\frac{2(n-1)(s+1)}{2s+1}}_{L^{\infty}(\Omega(S) )},
\end{align*}
and
\begin{align*}
   &\frac{C}{\sigma^{2 }} \| p \psi  \|_{L^2(\Omega )}  \|  u  \|_{L^2(\Omega(S))} \\  & \leqslant  
	\frac{C}{\sigma^2}\| \Dusm  (u \psi) \|^{\frac{2s}{1+s}}_{L^2(\Omega )} \|  u \psi  \|^{\frac{1-s}{1+s}}_{L^2(\Omega )} \|  u  \|_{L^2(\Omega(S))}
	+ \frac{C}{\sigma^{2(s+1)}} \|  u  \|_{L^2(\Omega )} \|  u  \|_{L^2(\Omega(S))} \\
	&\leqslant   \epsilon \| \Dusm  (u \psi) \|^2_{L^2(\Omega )} + \frac{C(\epsilon)}{\sigma^{2(s+1)}} \|  u \|^2_{L^2(\Omega(S) )} +
\frac{C }{\sigma^{2(s+1)}} \|  u  \|_{L^2(\Omega )} \|  u  \|_{L^2(\Omega(S))},
\end{align*}

Also,  using the interpolation inequality
\begin{equation}\label{ii-t}
\| u \|_{L^a(\Omega )} \leqslant   \| u\|_{L^c(\Omega )}^{\theta_2} \| u \|_{L^b(\Omega )}^{1-\theta_2} ,\qquad \text{ where } \theta_2 = \frac{c(a-b)}{a(c-b)},\ \quad
b < a < c,
\end{equation}
with $a = 2n, \, b = 2 , \, c = \frac{2d}{d-2s}$, and consequently $\theta_2 = \frac{d(n-1)}{2ns}$,
(using the embedding \cref{emb1} for $s<d/2$,
hence  $L^{\infty}((0,T); H^s(\Omega)) \subset L^{\infty}((0,T); L^{\frac{2d}{d-2s}} (\Omega)) $) we find that
\begin{align*}
&\frac{C}{\sigma^{2 }}
	\| p\psi  \|_{L^2(\Omega )}  \|  u  \|^n_{L^{2n}(\Omega(S))}\\
 & \leqslant  
	\frac{C}{\sigma^2}\| \Dusm  (u \psi) \|^{\frac{2s}{1+s}}_{L^2(\Omega )} \|  u \psi  \|^{\frac{1-s}{1+s}}_{L^2(\Omega )}\|  u  \|^n_{L^{2n}(\Omega(S))}
	+ \frac{C}{\sigma^{2(s+1)}} \|  u  \|_{L^2(\Omega )} \|  u  \|^n_{L^{2n}(\Omega(S))} \\
	&\leqslant   \epsilon \| \Dusm  (u \psi) \|^2_{L^2(\Omega )} + \frac{C(\epsilon)}{\sigma^{2(s+1)}}  \|  u \psi  \|^{ 1-s}_{L^2(\Omega )}\|  u  \|^{n(1+s)}_{L^{2n}(\Omega(S))}
	+ \frac{C }{\sigma^{2(s+1)}} \|  u  \|_{L^2(\Omega )} \|  u  \|^n_{L^{2n}(\Omega(S))} \\
	&\leqslant   \epsilon \| \Dusm  (u \psi) \|^2_{L^2(\Omega )} + \frac{C(\epsilon)}{\sigma^{2(s+1)}}  \|  u \psi  \|^{ 1-s}_{L^2(\Omega )}\|  u  \|^{n(1+s)(1-\theta_2)}_{L^{2}(\Omega(S))} \|  u  \|^{n(1+s) \theta_2}_{L^{\frac{2d}{d-2s}}(\Omega(S))} \\
	&\quad +
	\frac{C }{\sigma^{2(s+1)}} \|  u  \|_{L^2(\Omega )} \|  u  \|^{n (1-\theta_2)}_{L^{2}(\Omega(S))} \|  u  \|^{n \theta_2}_{L^{\frac{2d}{d-2s}}(\Omega(S))} \\
	&\leqslant    \epsilon \| \Dusm  (u \psi) \|^2_{L^2(\Omega )} + \frac{C(\epsilon)}{\sigma^{2(s+1)}}\|  u  \|^{ \frac{2s(1 -s) +(1+s)(2ns - d(n-1))}{2s} }_{L^{2}(\Omega(S))}   +
	\frac{C }{\sigma^{2(s+1)}} \|  u  \|_{L^2(\Omega )} \|  u  \|^{ \frac{ 2ns - d(n-1) }{2s} }_{L^{2}(\Omega(S))}  .
\end{align*}

By \cref{lem-fr} with $\alpha = \frac{s+1}{2}$ we have that
\begin{equation}\label{rrr-int}
	\| (-\Delta)^{\frac{s+1}{2}} ( u \psi_{S,\sigma })  \|^2_{L^2(\Omega)}
	\leqslant   2\| (-\Delta)^{\frac{s+1}{2}} ( u \psi_{S,\sigma })  \|^2_{L^2(\Omega(S+\sigma))} + \frac{C}{\sigma^{2(s+1)}}  \|   u \psi_{S,\sigma } \|^2_{L^2(\Omega )}.
\end{equation}
Taking into account that
$$\|u\psi_{S,\sigma}\|^2_{L^2(\Omega )}\leqslant   \|u\|^2_{L^2(\Omega(S))},$$
using these estimates and H\"older's inequalities,
from \cref{eq:weighted_entropy1} we arrive at
\begin{align*}
&\int_{\Omega(S+\sigma)}{G_{\varepsilon,\delta,\gamma}(u(x,T)) \,\d x}  +
 \iint_{\Omega_T(S +\sigma) }{ |(-\Delta)^{\frac{s+1}{2}}(u \,\psi_{S,\sigma } ) |^2   \,\d x\,\d t}  \\
 &\leqslant  
\int_{\Omega(S)}{G_{\varepsilon,\delta,\gamma}(u_{0,\varepsilon,\delta,\gamma}) \,\d x} +
\epsilon \iint_{\Omega_T(S+\sigma) }{ |(-\Delta)^{\frac{s+1}{2}}(u \,\psi_{S,\sigma } ) |^2   \,\d x \, \d t}
+ \frac{C }{\sigma^{2(s+1)}}   \int_{0}^T { \| u \|^{2}_{L^{2}(\Omega(S))} \,\d t} \\
&\quad +
\frac{C }{\sigma^{2(s+1)}} \Bigl(  \int_{0}^T { \| u \|^{2}_{L^{2}(\Omega(S))} \,\d t} \Bigr)^{\frac{s}{2s+1}}
 \Bigl(  \int_{0}^T { \| u \|^{2}_{L^{2}(\Omega)} \,\d t} \Bigr)^{\frac{s+1}{2s+1}}  \\
 &\quad +
\frac{C}{\sigma^{ 2(s+1) }} \Bigl(  \int_{0}^T { \|  u \|^{\frac{2}{1-(n-1)(s+1)}}_{L^2(\Omega(S) )} \,\d t} \Bigr)^{1 -(n-1)(s+1)}
 \Bigl(  \int_{0}^T { \|  u  \|^{2}_{L^{\infty}(\Omega(S) )}  \,\d t} \Bigr)^{ (n-1)(s+1) } \\
&\quad +
\frac{C }{\sigma^{2(s+1)}} \Bigl(  \int_{0}^T { \| u \|^{2}_{L^{2}(\Omega(S))} \,\d t} \Bigr)^{\frac{s}{2s+1}}
 \Bigl(  \int_{0}^T { \| u \|^{\frac{2}{2-n}}_{L^{2}(\Omega)} \,\d t} \Bigr)^{\frac{(s+1)(2-n)}{2s+1}}
\Bigl(  \int_{0}^T { \|  u  \|^{2}_{L^{\infty}(\Omega(S) )}  \,\d t} \Bigr)^{ \frac{(n-1)(s+1)}{2s+1} }
\\
&\quad +
\frac{C }{\sigma^{2(s+1)}}
\int_{0}^T { \|  u  \|^{ \frac{2s(1 -s) +(1+s)(2ns - d(n-1))}{2s} }_{L^{2}(\Omega(S))} \,\d t}  +
\frac{C }{\sigma^{2( s+1)}}
 \Bigl(  \int_{0}^T { \|  u    \|^{ \frac{2ns - d(n-1)}{s}} _{L^2(\Omega(S) )}\,\d t } \Bigr)^{ \frac{1}{2 } }
\Bigl(  \int_{0}^T {  \| u \|^{2} _{L^2(\Omega  )}\,\d t } \Bigr)^{\frac{1}{2 } } \\
&\leqslant  
\int_{\Omega(S)}{G_{\varepsilon,\delta,\gamma}(u_{0,\varepsilon,\delta,\gamma}) \,\d x} +
\epsilon \iint_{\Omega_T(S+\sigma) }{ |(-\Delta)^{\frac{s+1}{2}}(u \,\psi_{S,\sigma } ) |^2   \,\d x \,\d t}
+ \frac{C }{\sigma^{2(s+1)}}   \int_{0}^T { \| u \|^{2}_{L^{2}(\Omega(S))} \,\d t} \\
&\quad +
\frac{C }{\sigma^{2(s+1)}} \Bigl(  \int_{0}^T { \| u \|^{2}_{L^{2}(\Omega(S))} \,\d t} \Bigr)^{\frac{s}{2s+1}}
 \left(  \int_{0}^T { \| u \|^{2}_{L^{2}(\Omega)} \,\d t} \right)^{\frac{s+1}{2s+1}}  \\
 &\quad +
 \frac{C}{\sigma^{ 2(s+1) }} \Bigl(  \int_{0}^T { \|  u \|^{2 + \frac{2(n-1)(s+1)}{1-(n-1)(s+1)}}_{L^2(\Omega(S) )} \,\d t} \Bigr)^{1 -(n-1)(s+1)}
 \left(  \int_{0}^T { \|  u  \|^{2}_{L^{\infty}(\Omega(S) )}  \,\d t} \right)^{ (n-1)(s+1) } \\
 &\quad +
\frac{C }{\sigma^{2(s+1)}} \left(  \int_{0}^T { \| u \|^{2}_{L^{2}(\Omega(S))} \,\d t} \right)^{\frac{s}{2s+1}}
 \left(  \int_{0}^T { \| u \|^{2 + \frac{2(n-1)}{2-n}}_{L^{2}(\Omega)} \,\d t} \right)^{\frac{(s+1)(2-n)}{2s+1}}
\left(  \int_{0}^T { \|  u  \|^{2}_{L^{\infty}(\Omega(S) )}  \,\d t} \right)^{ \frac{(n-1)(s+1)}{2s+1} }\\
&\quad +
\frac{C }{\sigma^{2( s+1)}} T^{ \frac{(s+1)(n-1)(d-2s)}{4s}} \left(  \int_{0}^T { \|  u  \|^{ 2}_{L^{2}(\Omega(S))} \,\d t} \right)^{ \frac{2s(1 -s) +(1+s)(2ns - d(n-1))}{4s} }    \\
&\quad +
\frac{C }{\sigma^{2(s+1)}}T^{ \frac{(n-1)(d-2s)}{4s}}  \left(  \int_{0}^T { \|  u    \|^{2} _{L^2(\Omega(S) )}\,\d t } \right)^{ \frac{2ns - d(n-1)}{4s}}
\left(  \int_{0}^T {  \| u \|^{2} _{L^2(\Omega  )}\,\d t } \right)^{\frac{1}{2 } },
\end{align*}
where $1 < n < 1 + \frac{1}{s+1}$.
Using the embedding theorems   $L^{\infty}((0,T); H^s(\Omega)) \subset L^{\infty}((0,T); L^{\frac{2d}{d-2s}} (\Omega)) $,
$L^{2}((0,T); H^{s+1}(\Omega) ) \subset L^{2}(0,T ; L^{\infty}(\Omega) )$ with $\frac{d-2}{2}   < s < 1$, and \cref{rrr-int},
we obtain that
\begin{align}\label{rt-1}
\begin{aligned}
&\int_{\Omega(S+  \sigma)}{G_{\varepsilon,\delta,\gamma}(u(x,T)) \,\d x}  +
 \iint_{\Omega_T(S +\sigma) }{ |(-\Delta)^{\frac{s+1}{2}}(u \,\psi_{S,\sigma } ) |^2   \,\d x\,\d t}
 \\ &\leqslant    \int_{\Omega(S)}{G_{\varepsilon,\delta,\gamma}(u_{0,\varepsilon,\delta,\gamma}) \,\d x} +
\epsilon \iint_{\Omega_T(S+\sigma) }{ |(-\Delta)^{\frac{s+1}{2}}(u \,\psi_{S,\sigma } ) |^2   \,\d x\,\d t}
\\ &\qquad + \frac{C }{\sigma^{2(s+1)}}   \int_{0}^T { \| u \|^{2}_{L^{2}(\Omega(S))} \,\d t} +
\frac{C }{\sigma^{2(s+1)}}  \Bigl( \int_{0}^T { \| u \|^{2}_{L^{2}(\Omega(S))} \,\d t} \Bigr)^{\varpi},
\end{aligned}
\end{align}
where $\varpi \coloneqq \min \{ \frac{s}{2s+1}, 1 -(n-1)(s+1), \frac{2ns - d(n-1)}{4s} \} < 1 $.
Choosing, for instance,   $\epsilon =1/2$ in \cref{rt-1} we obtain that
\begin{align*}
	&\int_{\Omega(S+  \sigma)}G_{\varepsilon,\delta,\gamma}(u^N_{\varepsilon,\delta,\gamma}(x,T)) \,\d x  +
 \frac{1}{2}\iint_{\Omega_T(S +\sigma)} |(-\Delta)^{\frac{s+1}{2}}(u^N_{\varepsilon,\delta,\gamma} \,\psi_{S,\sigma } ) |^2   \,\d x\,\d t\\
 &\leqslant    \int_{\Omega(S)}{G_{\varepsilon,\delta,\gamma}(u^N_{0,\varepsilon,\delta,\gamma}) \,\d x} +
 \frac{C }{\sigma^{2(s+1)}}   \int_{0}^T { \| u^N_{\varepsilon,\delta,\gamma} \|^{2}_{L^{2}(\Omega(S))} \,\d t} +
\frac{C }{\sigma^{2(s+1)}}  \Bigl( \int_{0}^T { \| u^N_{\varepsilon,\delta,\gamma} \|^{2}_{L^{2}(\Omega(S))} \,\d t} \Bigr)^{\varpi}.
\end{align*}
Passing to the limit for $N\to+\infty$, $\gamma\to 0$ and after for $\delta\to 0$ and $\varepsilon\to 0$,
taking into account the convergences stated in the proof of \cref{Th-ex},
 we obtain \cref{rt-2} for $1 < n < 1 + \frac{1}{s+1}$ and $\frac{(d-2)_+}{2} < s < 1$.

\end{proof}

\section{Finite speed of propagations}\label{sec:FSP}

In this section, we use the entropy inequality \cref{rt-2} to deduce that the support of  solutions to \cref{eq:ft} propagates with finite speed.

\begin{proof}[Proof of \cref{th:speed}]

 Let us denote by
$$
A_T(S) \coloneqq  \iint_ {\Omega_T(S)}{u^2(x,t) \,\d x \,\d t}
$$
the energy of the solution $u$ on  $\Omega(S)$.
Then \cref{rt-2} can be written as
\begin{align}\label{rt-2-00}
\begin{aligned}
&\mathop {\sup}_{t \in (0,T)} \int_{\Omega(S+ \delta)}{G_0(u) \, \mathrm dx }   +
\frac12 \iint_{\Omega_T(S +\delta) }{ |(-\Delta)^{\frac{s+1}{2}}(u \,\psi_{S,\delta } ) |^2   \,\d x\,  \d t}  \\&\leqslant  
 \int_{\Omega(S)}{G_0(u_0) \,\d x} +
 \frac{C}{\delta^{2(s+1) }} A_T(S) + \frac{C}{\delta^{2(s+1) }} A_T^{\varpi}(S)  \eqqcolon  R_T(S,\delta).
 \end{aligned}
\end{align}

We will prove the property of finite propagation speed by contradiction. Specifically, assume that for all
$t \in (0,T]$, the support of $u(\cdot, t)$ is $\bar{\Omega}$, i.\,e., $\operatorname{supp}u(\cdot, t) = \bar{\Omega}$
for every $t \in (0,T]$. Therefore, we assume that $u(\cdot, t)>0$  in $\bar{\Omega}$ and for any  $t\in (0,T]$.
The argument will show that there exists $T^{**}>0$ and a function $d=d(T)<R$ such that
$$
A_T(S) = 0\qquad\text{for all }T\in [0,T^{**}] \text{ and for all }  S\geqslant   d(T).
$$
This is clearly in contradiction with the assumed $u >0$ in $\bar{\Omega}$.

The entropy inequality \cref{eq:entropy} and Sobolev embeddings give that the solution $u$ is {H\"{o}lder continuous in $\bar{\Omega}$ for a.\,e. $t \in (0,T)$. Thus, by the Weierstrass extreme value theorem and our assumption, there exists $ c(t) \coloneqq \mathop {\min}_{x \in \bar{\Omega}} u(t,x) \in L^2(0,T)$ such that
$$
u(\cdot, t)\geqslant  c(t) > 0 \qquad \text{ in }  \bar{\Omega} \text{ for  a.\,e. } t \in (0,T].
$$ }
Consequently,
$$
A_T(S) \geqslant  \frac{\pi^{\frac{d}{2}}}{\Gamma( \frac{d}{2} + 1 )}R^{d-1}(R-S)
\int_0^T {{c^2(t)} \,dt} \eqqcolon  C_0(T) (R - S).
$$

Let $\gamma_T:[0,R]\to [0,+\infty)$ be a non-increasing function such that
\begin{equation}\label{ddr}
\begin{cases}
0 < \gamma_T(0) \leqslant   C_0(T) R, \\
\gamma_T(S) \leqslant   C_0(T) (R-S)  &\text{for all } S\in[0,R],\\
\gamma_T(S) > 0 & \text{for all }  S\in[0,R),\\
\gamma_T(S + \delta)  \leqslant   \kappa  \gamma_T^{\beta + \frac{\beta}{\lambda} }(S)  &\text{for all }  S \geqslant  0, \ \delta > 0,
\end{cases}
\end{equation}
with $0 < \gamma_T(0) < \min\{1, C_0(T) \}R $, where $\lambda>0$ is a free parameter while $\beta$ and $\kappa$ verify
\begin{equation} \label{eq:parameters}
\beta\coloneqq  1+\frac{n(1-\theta) }{ 2-n  } > 1, \qquad
\kappa \in \Bigl(0,  \frac{1 - \frac{\gamma_T(0)}{R} }{ \gamma_T^{\beta -1 + \frac{\beta}{\lambda}}(0)} \Bigr).
\end{equation}
We observe that a function satisfying the assumptions \cref{ddr} exists, see \cref{ex:examplegamma}.
Thanks to \cref{lem-st-n} the function $\gamma_T$ also verifies $\gamma_T(S) \equiv 0 $ for all $S \geqslant  R $.
As a result, we have
\begin{equation}\label{as-b}
 A_T(S) \geqslant  \gamma_T(S )   \qquad\text{for all } T >0, \  S \in [0,R],
\end{equation}
and, thus, for any $S\in [0,R]$ and for any $T>0$, calling $\nu \coloneqq 1-\varpi>0$,
\begin{equation}\label{eq:A_T}
A_T(S) +  A_T^{\varpi}(S) \leqslant   \frac{\max\left\{A_T^{\nu}(0),1\right\}}{\gamma^{\nu}_T(S)} A_T(S) \eqqcolon \tilde{C}_T(S) A_T(S).
\end{equation}
Estimate \cref{rt-2-00} becomes
\begin{equation}\label{e-2-004}
\begin{aligned}
&{\mathop {\sup}_{t \in (0,T)} \int_{\Omega(S+ \delta)}{G_0(u) \, \mathrm dx } }  +
\frac{1}{2} \iint _{\Omega_T(S+ \delta)}{ \bigl|(-\Delta)^{\frac{s+1}{2}}(u \,\psi_{S,\delta } )  \bigr|^2   \, \mathrm dx \, \mathrm dt}
\\ & \leqslant   \int_{\Omega(S)}{G_0(u_0) \,\mathrm dx} +
\frac{\tilde{C}_T(S)}{ \delta^{2(s+1)  } } A_T(S) .
\end{aligned}
\end{equation}

{Now, we show that for $n \in (1,2)$ there is a positive constant $C$ depending on $n$ only
such that
\begin{equation}\label{entr-n}
 \int_{\Omega(S)}{  u^{2-n}  \, \mathrm dx }  \leqslant   C\, R_T(S,\delta).
\end{equation}
We start with the inequality
$$
y(S) \coloneqq  \int_{\Omega(S)}{ u^{2-n} \, \mathrm dx }  \leqslant   |\Omega(S)|^{n-1} \Bigl( \int_{\Omega(S)}{ u  \, \mathrm dx } \Bigr)^{2-n},
$$
from which we deduce that
$$
 \int_{\Omega(S)}{G_0(u) \, \mathrm dx } \geqslant  F(y(S)) \coloneqq\frac{|\Omega(S)|^{- \frac{n-1}{2-n}}}{n-1} y^{\frac{1}{2-n}}(S) -
 \frac{1}{(2-n)(n-1)} y(S) + \frac{|\Omega(S)|}{2-n}.
$$
The function
\[
F(y) = \frac{|\Omega(S)|^{- \frac{n-1}{2-n}}}{n-1} y^{\frac{1}{2-n}} -
 \frac{1}{(2-n)(n-1)} y + \frac{|\Omega(S)|}{2-n}
\]
is non-negative and convex for $y\ge 0$. Moreover, it satisfies that $F(\abs{\Omega(S)})=F'(\abs{\Omega(S)}) = 0$ and
$F''(\abs{\Omega(S)}) = \frac{1}{(2-n)^2} |\Omega(S)|^{- 1}$. Therefore
$$
F(y) \geqslant  \frac{1}{2(2-n)^2} |\Omega(S)|^{- 1} (y -   |\Omega(S)| )^2 \qquad \text{for all } y\ge 0.
$$
}
Using this lower bound, we have
$$
\left( \int_{\Omega(S)}{ (u^{2-n} -1 ) \, \mathrm dx }\right)^{2}  \leqslant   2(2-n)^2 |\Omega(S)|
  \int_{\Omega(S)}{G_0(u) \, \mathrm dx },
$$
whence
\begin{align*}
 \int_{\Omega(S)}{ u^{2-n} \, \mathrm dx }   &\leqslant   \sqrt{2}(2-n) \Bigl( |\Omega(S)|
  \int_{\Omega(S)}{G_0(u) \, \mathrm dx } \Bigr)^{\frac{1}{2}} +   |\Omega(S)|\\
   &=     
 \sqrt{2}(2-n)^{\frac{3}{2}} \Bigl( \int_{\Omega(S)}{G_0(0) \, \mathrm dx }
  \int_{\Omega(S)}{G_0(u) \, \mathrm dx } \Bigr)^{\frac{1}{2}} +   (2-n)\int_{\Omega(S)}{G_0(0) \, \mathrm dx }\\
  & \leqslant   
2(2-n) \max \{ 1, \sqrt{2(2-n)}\} \max \Bigl\{ \int_{\Omega(S)}{G_0(u) \, \mathrm dx }, \int_{\Omega(S)}{G_0(0) \, \mathrm dx } \Bigr\}.
\end{align*}
As a result, by \cref{rt-2-00}, we obtain \cref{entr-n}.

Applying the Gagliardo--Nirenberg interpolation inequality (see \cref{G-N-nn})
in the region $\Omega$ to the function $v(\cdot,t) \coloneqq  u(\cdot,t)  $ with $ b =
2-n$  and $\theta  = \frac{nd}{nd + 2 (2-n)(s+1)}$, after
integrating in time and using \cref{eq:entropy}, we get
$$
A_T(0) \leqslant   C\,T^{\mu} \|u_0\|^{(2-n)\beta}_{L^{2-n}(\Omega)} + C\,T  \|u_0\|^{2}_{L^{2-n}(\Omega)},
$$
whence
\begin{equation}\label{a-in}
A_T(0) \leqslant   C\,T^{\mu} \|u_0\|^{(2-n)\beta}_{L^{2-n}(\Omega)} {\text{ for small enough }} T > 0.
\end{equation}

Next, using the homogeneous Gagliardo--Nirenberg interpolation inequality (see \cref{G-N-nn})
in the region $\Omega(S+  \delta)$ with the function $v(\cdot,t) \coloneqq  u(\cdot,t)\psi _{S+\delta,\delta }(\cdot) $ with $b =
2-n$  and $\theta  = \frac{nd}{nd + 2 (2-n)(s+1)}$, we find that
\begin{equation}\label{rrr-ll}
  \|   u \psi _{S+\delta,\delta } \|^2_{L^2(\Omega(S+\delta))}  \leqslant  
C \|  (-\Delta)^{\frac{s+1}{2}} ( u \psi _{S+\delta,\delta } ) \|^{ 2\theta }_{L^{2}  (\Omega)}  \| u \psi _{S+\delta,\delta } \|^{ 2 (1-\theta) }_{L^{2-n}(\Omega(S+ \delta))}.
\end{equation}
From \cref{rrr-ll}, using \cref{rrr-int}, it follows that
\begin{align*}
  \|   u \psi _{S+\delta,\delta } \|^2_{L^2(\Omega(S+\delta))} &\leqslant  
C \|  (-\Delta)^{\frac{s+1}{2}} ( u \psi _{S+\delta,\delta } ) \|^{ 2\theta }_{L^{2}  (\Omega(S+2\delta))}  \| u \psi _{S+\delta,\delta } \|^{ 2 (1-\theta) }_{L^{2-n}(\Omega(S+ \delta))}    \\
& +( \frac{C}{S + \delta})^{ 2\theta(s+1 ) } \|    u \psi _{S+\delta,\delta }  \|^{ 2\theta  }_{L^{2}  (\Omega(S+ \delta))}  \| u \psi _{S+\delta,\delta } \|^{ 2 (1-\theta) }_{L^{2-n}(\Omega(S+ \delta))} ,
\end{align*}
whence, using Young's inequality, we deduce that
\begin{equation} \label{int-3}
\begin{aligned}
\|   u  \|^2_{L^2(\Omega(S+ 2\delta))} & \leqslant  
C \|  (-\Delta)^{\frac{s+1}{2}} ( u \psi _{S+\delta,\delta } ) \|^{ 2\theta  }_{L^{2}  (\Omega(S+ 2\delta))}  \| u  \|^{ 2 (1-\theta) }_{L^{2-n}(\Omega(S+ \delta))}
 \\ &  + \frac{C}{(S + \delta)^{\frac{2\theta(s+1)}{ 1-\theta  }}} \| u  \|^{2}_{L^{2-n}(\Omega(S+ \delta))} .
\end{aligned}
\end{equation}
Integrating  \cref{int-3} with respect to time, using H\"older's inequality, \cref{rt-2-00} and
\cref{entr-n}, exchanging $S + \delta$
with $S$, we arrive at the following relation:
$$
A_T(S + \delta) \leqslant   C\,T^{ 1-\theta  } R^{1+\frac{n(1-\theta) }{ 2-n  }}_T(S,\delta)
+ \frac{C}{(S + \delta)^{\frac{2\theta(s+1)}{ 1-\theta  }}} T  R^{1+\frac{n}{ 2-n}}_T(S,\delta).
$$

Now, we consider $S\geqslant  r_0$ which implies that $u_0\equiv 0$ on $\Omega(S)$. Thus $G_0(u_0) = G_0(0)$. Consequently,
\begin{align}\label{eq:estg0}
{\mathfrak{G}(s) \coloneqq}\int_{\Omega(S)} G_0(u_0) \, \mathrm d x  = \frac{ | \Omega(S)|}{2-n}
\leqslant    \frac{d \,\pi^{\frac{d}{2}}}{(2-n)\Gamma( \frac{d}{2} + 1 )} R^{d - 1}  (R - S)
\leqslant   {\frac{d}{(2-n)C_0(T)} A_T(S)}
\end{align}
for all $S \geqslant  r_0$.  We recall that
$$
R_T(S,\delta) \leqslant   \frac{C}{ \delta^{2(s+1)} } (A_T(0) + A_T^{\varpi}(0)) {+ \mathfrak{G}(0)}.
$$
Therefore, we can find $T^*=T^*(\delta)$ such that, for $T\leqslant   T^*$ (and $S\geqslant  r_0$), 
\begin{equation}\label{eq:fsp_energy1}
A_T(S + \delta) \leqslant   C\,T^{ 1-\theta  } R^{1+\frac{n(1-\theta) }{ 2-n  }}_T(S,\delta).
\end{equation}
Moreover, due to \cref{eq:A_T}, we deduce that
\begin{equation}\label{e-12}
A_T(S +\delta) \leqslant    C\, T^{\mu} \bigl(\tilde{C}_T(S) \delta^{- \tilde{\alpha} } A_T(S) {+ \mathfrak{G}(S)} \bigr)^{\beta} \text{ for all } S \geqslant  r_0,
\end{equation}
where we recall (see \cref{eq:parameters})
$$
\mu =  1-\theta  , \ \tilde{\alpha} = 2(s+1) , \
\beta = 1+\frac{n(1-\theta) }{ 2-n  } > 1.
$$
Consequently, calling $\tilde{A}_T(S ) \coloneqq  \tilde{C}_T^{ - \lambda}(S) A_T(S)$, we have
$$
\tilde{A}_T(S +\delta) \leqslant    C\, \kappa^{\lambda \nu}  T^{\mu} \bigl( \delta^{-\tilde{\alpha}} \tilde{A}_T(S)
{+ \mathfrak{G}(S)\tilde{C}_T^{ - \lambda -1}(S) }  \bigr)^{\beta}
\text{ for all } S \geqslant  r_0.
$$
By \cref{eq:estg0},
$$
\mathfrak{G}(S)\tilde{C}_T^{ - \lambda -1}(S) =
\mathfrak{G}(S)\tilde{C}_T^{-1}(S) \frac{\tilde{A}_T(S)}{A_T(S)} \leqslant   C\,
\frac{\gamma_T^{\nu }(S) }{C_0(T)}\tilde{A}_T(S) \leqslant  
\frac{C}{ C^{1-\nu}_0(T)} \tilde{A}_T(S) \leqslant  
\frac{C }{\delta^{ \tilde{\alpha}}}\tilde{A}_T(S)
$$
for small enough $\delta > 0$. As a result, we obtain
\begin{equation}\label{e-12-000}
\tilde{A}_T(S +\delta) \leqslant    C_4 \kappa^{\lambda \nu}  T^{\mu} \bigl( \delta^{-\tilde{\alpha}} \tilde{A}_T(S) \bigr)^{\beta}
\text{ for all } S \geqslant  r_0.
\end{equation}
Therefore, we are in the position to apply of the Stampacchia's iteration lemma (in the form contained in \cref{L-1}) to \cref{e-12-000} with $\alpha = \tilde{\alpha} \beta$, $\beta  > 1$, and $C = C_4 \kappa^{\lambda \nu}  T^{\mu}$.
We obtain that
\begin{equation}\label{e-13}
\tilde{A}_T(S) = 0  \qquad\text{for all } S \geqslant  d(T) ,
\end{equation}
where
\[
d(T) = r_0 + 2^{\frac{\beta}{\beta -1}} (C_4 \kappa^{\lambda \nu}  T^{\mu}  \tilde{A}_T^{\beta -1}(r_0) )^{\frac{1}{\tilde{\alpha}\beta}}.
\]
{As
$$
\tilde{A}_T(r_0) \leqslant   C\, \gamma_T^{\lambda \nu}(r_0) A_T(r_0) \leqslant  
C\,\gamma_T^{\lambda \nu}(0) A_T(0) \leqslant   C \, \gamma_T^{\lambda \nu}(0) T^{\mu} \|u_0\|_{2-n}^{(2-  n)\beta}
$$
then, taking $ 0 < \kappa \leqslant    \gamma_T^{1-\beta}(0)$, an upper bound for $d(T)$ is
$$
d(T) \leqslant   r_0 + C_5\, T^{\frac{\mu }{\tilde{\alpha}}   } =
r_0 + C_5 T^{\frac{1-\theta}{  2(s+1)  }  } .
$$
For $T \leqslant   T^{**}\coloneqq \min \bigl\{T^*, \bigl(\frac{R-r_0 }{C_5}\bigr)^{\frac{2(s+1)}{ 1-\theta } }\bigr\}$ we have that
$$
C_5 \, T^{\frac{  1-\theta  }{  2(s+1)  }  }  < R - r_0.
$$}
Thus, for $T\leqslant   T^{**}$ we have $d(T) <R$.
On the other hand, by \cref{as-b} we have
$$
\tilde{A}_T(S) \geqslant  C \, \gamma_T^{1 + \lambda\nu}(S) > 0 \text{ for all } S \in [0,R).
$$
So, since $\tilde{A}_T(S) = 0$ for all $S \geqslant  d(T) $,  and  $d(T) <  R $  for all {$T \in [0,T^{**}]$}, we obtain the contradiction to \cref{as-b}.

Next, we look for an exact estimate of interface speed.
Using the homogeneous Gagliardo--Nirenberg interpolation inequality (see \cref{G-N-nn})
in the region $\Omega(S+\delta)$ with the function $v(\cdot,t) \coloneqq  u(\cdot,t)\psi _{S+\delta,\delta }(.) $ with $ b =
1 $  and $\theta  = \frac{d}{d+2(s+1)}$, we find that
\begin{align*}
\| u \psi _{S+\delta,\delta }  \|^2_{L^2(\Omega(S+ \delta))} &\leqslant    C\, M^{\frac{4(s+1)}{d+ 2(s+1)}}  \|(-\Delta)^{\frac{s+1}{2}} ( u \psi _{S+\delta,\delta } ) \|^{\frac{2d}{d+ 2(s+1)}}_{L^{2}  (\Omega)}  \\ &\leqslant  
C\, M^{\frac{4(s+1)}{d+ 2(s+1)}}  \|(-\Delta)^{\frac{s+1}{2}} ( u \psi _{S+\delta,\delta } ) \|^{\frac{2d}{d+ 2(s+1)}}_{L^{2}  (\Omega(S+ 2\delta))} \\ \qquad +  \frac{C M^{\frac{4(s+1)}{d+ 2(s+1)}}}{(S+ \delta)^\frac{2d(s+1)}{d+ 2(s+1)} } \| u \psi _{S+\delta,\delta }  \|^{\frac{2d}{d+ 2(s+1)}}_{L^{2}  (\Omega(S+ \delta))}  ,
\end{align*}
whence
\begin{equation}\label{int-5}
\| u  \|^2_{L^2(\Omega(S+2\delta))}   \leqslant    C\, M^{\frac{4(s+1)}{d+ 2(s+1)}}  \|(-\Delta)^{\frac{s+1}{2}} ( u \psi _{S+\delta,\delta } ) \|^{\frac{2d}{d+ 2(s+1)}}_{L^{2}  (\Omega(S+ 2\delta))} +  \frac{C M^{2}}{ (S+ \delta)^{d}} .
\end{equation}
Integrating  \cref{int-5} with respect to time, using H\"older's inequality and \cref{e-2-004},
exchanging $2\delta$ with $\delta$, we arrive at the following
relation:
\begin{equation}\label{e-12-2}
A_T(S + \delta) \leqslant    C_6(M) T^{\frac{2(s+1)}{d +2(s+1)}} \bigl( \delta^{-\tilde{\alpha}} \tilde{C}_T(S)  A_T(S) \bigr)^{ \frac{d}{d +2(s+1)}}
\end{equation}
for $S > 0$ and enough small $T >0$. Hence, by \cref{e-12-2} for
$ \tilde{A}_T(S ) \coloneqq  \tilde{C}_T^{ - \lambda }(S) A_T(S) $ we have
\begin{equation}\label{e-12-222}
\tilde{A}_T(S +\delta) \leqslant    C_7(M)  \kappa^{\lambda \nu}
 T^{\frac{2(s+1)}{d +2(s+1)} } \bigl( \delta^{- \tilde{\alpha}} \tilde{A}_T(S) \bigr)^{ \frac{d}{d +2(s+1)} }
\end{equation}
for $S > 0 $. Applying \cref{L-1}
to \cref{e-12-222}
with ${\alpha} =\frac{d\tilde{\alpha}}{d +2(s+1)}  $, ${\beta} =  \frac{d}{d +2(s+1)}  < 1$, and $C =  C_7(M)  \kappa^{\lambda \nu} T^{\frac{2(s+1)}{d +2(s+1)} } $, we find that
\begin{equation}\label{int-6}
\tilde{A}_T(S) \leqslant   C_8(M) T \, S^{- \frac{ d \tilde{\alpha}}{2(s+1)}},  \qquad\text{for all } S > 0,
\end{equation}
where $C_8(M) = 2^{\frac{d \tilde{\alpha} (d+2(s+1))}{4(s+1)^2}}C_7^{\frac{ d+2(s+1) }{2(s+1)}}(M)  \kappa^{ \frac{ \lambda \nu ( d+2(s+1)) }{2(s+1)} }$.
From \cref{e-13}, using  \cref{int-6} with $S = \Gamma(T) \coloneqq  d(T)- r_0 \leqslant   r_0$, we arrive at
$$
\Gamma(T) \leqslant    2^{\frac{\beta}{\beta -1}} \bigl[C_4 T^{\mu}
 \kappa^{\lambda\nu} \tilde{A}_T^{\beta -1}(\Gamma(T)) \bigr]^{\frac{1}{ \tilde{\alpha}\beta}}
\leqslant      C_9(M) T^{\frac{\mu + \beta -1}{\tilde{\alpha} \beta}} (\Gamma(T))^{-\frac{ d( \beta -1 )}{ 2\beta(s+1) }},
$$
where $ C_9(M) = 2^{\frac{\beta}{\beta -1}}C_4^{\frac{1}{\tilde{\alpha} \beta}}  \kappa^{\frac{\lambda\nu}{\tilde{\alpha} \beta}}
C^{\frac{\beta -1}{ \tilde{\alpha}\beta}}_8(M) $. So, {taking $0 < \kappa \leqslant   1$,
$$
 \Gamma(T) \leqslant  
  C^{\frac{2 \beta (s+1)}{ 2\beta(s+1) + d(\beta-1)}}_9(M)
T^{ \frac{2(s+1)(\mu + \beta -1) }{ \tilde{\alpha} ( 2\beta(s+1) + d(\beta-1) ) } } \leqslant  
C_{10}(M) T^{ \frac{2(s+1)}{ \tilde{\alpha} (  n d + 2(s + 1))}} =
C_{10}(M) T^{ \frac{1}{   nd + 2(s+1)   }} ,
$$
where
$$
C_{10}(M) = \bigl[ 2^{\frac{\beta}{\beta -1}}C_4^{\frac{1}{\tilde{\alpha} \beta}} \bigl( 2^{\frac{d \tilde{\alpha} (d+2(s+1))}{4(s+1)^2}}C_7^{\frac{ d+2(s+1) }{2(s+1)}}(M) \bigr)^{\frac{\beta -1}{ \tilde{\alpha}\beta}}  \bigr]^{\frac{2 \beta (s+1)}{ 2\beta(s+1) + d(\beta-1)}}.
$$ }
As a result, we find that
$$
d(T) \leqslant    r_0 + C_{10}(M) T^{ \frac{1}{   nd + 2(s+1)   }}  \qquad\text{for all } T \leqslant   T^*.
$$
\end{proof}

\section{Lower-bound on the waiting time} \label{sec:LWT}

We now turn to the proof of the optimal lower bounds on waiting times.

\begin{proof}[Proof of \cref{th:wt}]
We assume that $\supp u_0 \subset \Omega \setminus \Omega(r_0)$.
We prove that there exists a time $T_0\in (0,T^{*}]$ such that we have
\begin{align*}
u(\cdot, t) = 0 \quad \text{ in } \Omega(r_0) \quad \text{for}\quad 0 < t < T_0.
\end{align*}
With the same notations as in the proof of \cref{th:speed}, we recall that we have the refined entropy inequality
\begin{equation*}
\begin{aligned}
&\int_{\Omega( S + \delta)}{G_0(u) \dd x}  +
C \iint_{\Omega_T(S+\delta)} { | (-\Delta)^{\frac{s+1}{2}} (u \, \psi_{S ,\delta}) |^{2} \dd x \dd t}  \\
&\leqslant   C \int_{\Omega( S)}{G_0(u_0) \dd x} +
\frac{\tilde{C}_T(S)}{\delta^{2(s+1)}}   \int_{0}^T { \| u \|^{2}_{L^{2}(\Omega( S)) } \dd t}.
\end{aligned}
\end{equation*}
Arguing as in \cref{eq:estg0}, {taking $S\geqslant  r_0-\delta$ and $0 < \kappa \leqslant   1$, we arrive at
\begin{align*}
\begin{aligned}
\tilde{A}_T(S+\delta) &\leqslant    C_4   T^{\mu} \left( \delta^{- 2(s+1) } \left( \int_{\Omega(S)} \delta^{2(s+1)} |G_0(u_0)-G_0(0)| \dd x  + \tilde{A}_T(S) \right) \right)^{\beta} \\
& \leqslant   C_4 T^{\mu} \left( \delta^{-2(s+1)} \bigl(  \mathcal{S} \, \delta^{\sigma}  + \tilde{A}_T(S) \bigr) \right)^{\beta},
\end{aligned}
\end{align*}
where we introduced the notation
\begin{align*}
&\mu =   1-\theta , \quad \theta  = \frac{nd}{nd + 2 (2-n)(s+1)}, \quad
\beta  =  1 + \frac{n(1-\theta) }{ 2-n  }, \quad \sigma  = d + 2(s+1) + (2-n)\gamma, \\
& \mathcal{S} \coloneqq  \sup_{\delta>0} \delta^{-\gamma(2-n)} \fint_{ B_{r_0}(0) \setminus B_{r_0 -\delta}(0) } |G(u_0(x))-G_0(0)| \dd x.
\end{align*} }

Arguing as in \cite{GiacomelliGruen}, by relying on Stampacchia's lemma (see  \cref{lem:stampacchia2}),
we have that there exists a time $T_0$ such that
\begin{align*}
\tilde{A}_{T_0} (r_0) \equiv 0  \Rightarrow  \int_0^{T_0} \int_{ \Omega(r_0)} u^2(t,x) \dd x \dd t \equiv 0
\end{align*}
provided
$$
d + 2(s+1) + (2-n)\gamma \geqslant   \frac{2\beta(s+1)}{\beta-1},
$$
i.\,e., $\gamma \geqslant   \frac{2(s+1)}{n}$,  and that the following lower bound holds:
\begin{align*}
T_0 \geqslant  C \mathcal{S}^{-\frac{\beta-1}{\mu}} = C \mathcal{S}^{-\frac{n}{2-n}}.
\end{align*}

\end{proof}

\appendix
\section{Technical lemmas} \label{app:lemmas}

In this appendix, we collect some technical results that have been used throughout the paper.

\subsection{Interpolation inequalities and iteration lemmas}

We start with an interpolation inequality of Gagliardo--Nirenberg type (inspired by \cite[Proposition A.1]{DalPassoGiacomelliShishkov}).

\begin{lemma}[Gagliardo--Nirenberg-type inequality] \label{G-N-nn}
 If $\Omega  \subset \mathbb{R}^N $ is a bounded
domain with piecewise-smooth boundary, $b \in (0, 2)$  and $s \in (0,1)$, then there
exist  positive constants  $C_1$ and $C_2$ ($C_2 =0$ if $\Omega$ is unbounded or $v(x)$ has a compact support)
depending only on $\Omega ,\ s,\ b,$
and $N$ such that the following inequality is valid for every
$v(x) \in H^{s+1}(\Omega ) \cap L^b (\Omega )$:
$$
\| v  \|_{L^2 (\Omega )}  \leqslant   C_1  \|
(-\Delta)^{\frac{s+1}{2}} v  \|_{L^2 (\Omega )}^\theta  \| v  \|_{L^b
(\Omega )}^{1 - \theta } + C_2 \| v  \|_{L^b
(\Omega )}, \quad  \theta  = \frac{{\frac{1} {b}   - \frac{1}
{2}}} {{\frac{1} {b} + \frac{s+1} {N} - \frac{1} {2}}} \in  [0,1 ) .
$$
\end{lemma}

\begin{proof}[Proof of \cref{G-N-nn}]
Using the interpolation inequality from \cite[Proposition A.1]{DalPassoGiacomelliShishkov}
\begin{equation}\label{g-n-0}
\| v  \|_{L^a (\Omega )}  \leqslant   d_1  \|
D^j v  \|_{L^d (\Omega )}^{\theta_1}  \| v  \|_{L^b
(\Omega )}^{1 - \theta_1 } + d_2 \| v  \|_{L^b
(\Omega )} , \quad \theta_1  = \frac{{\frac{1} {b}   - \frac{1}
{a}}} {{\frac{1} {b} + \frac{j} {N} - \frac{1} {d}}} \in  [0,1 ) ,
\end{equation}
where $a > 0$, $b \in (0,a)$, $j > 0$, and $\frac{1}{d} - \frac{j}{N} < \frac{1}{a}$. Moreover,
$d_2 = 0$ if $\Omega$ is unbounded or $v(x)$ has a compact support.

Using \cref{g-n-0} with $a = d =2$, $j =1$, and $b \in (0,2)$, we have
\begin{equation}\label{g-n-1}
\| v  \|_{L^2 (\Omega )}  \leqslant   d_1  \|
\nabla v  \|_{L^2 (\Omega )}^{\theta_1}  \| v  \|_{L^b
(\Omega )}^{1 - \theta_1 } + d_2 \| v  \|_{L^b
(\Omega )} , \quad \theta_1  = \frac{{\frac{1} {b}   - \frac{1}
{2}}} {{\frac{1} {b} + \frac{1} {N} - \frac{1} {2}}} \in  [0,1 ) .
\end{equation}
Since $\| \nabla v \|_{L^2 (\Omega )} = \| (-\Delta)^{\frac{1}{2}} v  \|_{L^2 (\Omega )}$, then,
using \cref{dong-01} with $\beta = \frac{1}{2}$, we get
\begin{equation}\label{g-n-2}
\| \nabla v  \|_{L^2(\Omega)} \leqslant    \| (-\Delta)^{\frac{s+1}{2}} v \|_{L^2(\Omega)}^{\theta_2} \| v \|_{L^2(\Omega)}^{1-\theta_2},
\quad \theta_2 = \frac{1}{s+1}.
\end{equation}
Plugging \cref{g-n-2} into  \cref{g-n-1}, we obtain
$$
\| v  \|_{L^2 (\Omega )}  \leqslant   d_1  \|
(-\Delta)^{\frac{s+1}{2}} v \|_{L^2 (\Omega )}^{\theta_1 \theta_2 } \| v \|_{L^2(\Omega)}^{\theta_1(1-\theta_2)}  \| v  \|_{L^b
(\Omega )}^{1 - \theta_1 } + d_2 \| v  \|_{L^b (\Omega )}.
$$
Applying Young's inequality with exponents $\frac{1}{\theta_1(1-\theta_2)}$ and $\frac{1}{1 - \theta_1(1-\theta_2)}$, we deduce that
$$
\| v  \|_{L^2 (\Omega )}  \leqslant   \tilde{d}_1  \|
(-\Delta)^{\frac{s+1}{2}} v \|_{L^2 (\Omega )}^{\frac{\theta_1 \theta_2 }{1 - \theta_1(1-\theta_2)}}  \| v  \|_{L^b
(\Omega )}^{\frac{1 - \theta_1}{1 - \theta_1(1-\theta_2)} } + \tilde{d}_2 \| v  \|_{L^b (\Omega )}.
$$
This completes the proof.
\end{proof}

Next, we recall Stampacchia's classical iteration lemma (see \cite[Lemme 4.1,p.14]{S1}) and an inhomogeneous version that has been introduced in \cite[Lemma 3.1]{DalPassoGiacomelliGruen}, \cite[Lemma 2.4]{GiacomelliGruen}, and \cite[Lemma 4]{MR1642176} (stated below as in \cite{GiacomelliGruen}). These are fundamental tools in studying interface evolution properties for degenerate parabolic equations via energy methods.

\begin{lemma}[Classical Stampacchia's lemma]\label{L-1}
Let $f(x)  $ be non-negative, non-increasing in $[x_0,+\infty)$ function. Assume that $f$ satisfies
\begin{equation}\label{s-1}
 f(y) \leqslant   \frac{C}{(y -x)^{\alpha}} f^{\beta}(x)  \quad  \text{ for } y > x \geqslant  x_0,
\end{equation}
where $C,\,\alpha,\, \beta$ are some positive constants. Then

(i) if  $\beta > 1$ we have
$$
f(y) = 0  \quad \text{ for all }  y \geqslant  x_0 + d,
$$
where $d^{\alpha} = C f^{\beta -1}(x_0) 2^{\frac{\alpha\beta}{\beta -1}}$;

(ii) if  $\beta = 1$ we get
$$
f(y) \leqslant   e^{1- \zeta(y -x_0)}f(x_0) \quad \text{ for all }  y \geqslant  x_0,
$$
where $ \zeta = (e\,C)^{- \frac{1}{\alpha}}$;

(iii) if  $\beta < 1$ we obtain
$$
f(y) \leqslant    2^{\frac{\mu}{1-\beta}} \bigl[C^{\frac{1}{1-\beta}} + (2\, x_0)^{\mu} f(x_0) \bigr] y^{-\mu} \quad \text{ for all }  y \geqslant  x_0 > 0,
$$
where $ \mu =  \frac{\alpha }{1- \beta}$.
\end{lemma}

\begin{lemma}[Stampacchia-type lemma]\label{lem-st-n}
Let  $f : [0,+\infty) \to [0, +\infty)$ be a non-negative non-increasing function
such that
\begin{equation}\label{nnn-1}
f(s +\delta) \leqslant   \epsilon f^{\nu}(s) \text{ for all } s \in \mathbb{R}^+, \ \delta > 0,
\end{equation}
for  $\epsilon \in (0,f^{1-\nu}(0))$ and $\nu > 1$. Then
$$
f(s) \equiv 0 \text{ for all } s \geqslant  d\coloneqq  \frac{f(0)}{1 - \epsilon f^{\nu -1}(0)}.
$$
\end{lemma}

\begin{example}\label{ex:examplegamma}
A class of functions
$$
f(s) = \frac{f(0)}{d}  e^{\, \xi(s) \,\nu^{\frac{s}{\delta}}},
$$
for any non-positive function $\xi(s)$ such that
$$
\xi(s+\delta) \leqslant   \xi(s), \quad \text{ with } \xi(0) = \ln ( \epsilon f^{\nu -1}(0))^{\frac{1}{\nu -1}} \leqslant   0
\text{ and }  \xi(d^-) = - \infty,
$$
satisfies \cref{nnn-1} and $f(s) \equiv 0$ for all $s \geqslant  d  = (\epsilon f^{\nu-1}(0))^{\frac{1}{\nu-1}} \in (0,1]$.
For example, if $\xi(s) = \ln (d-s)_+$ then
$f(s) = \frac{f(0)}{d}  (d-s)_+^{\nu^{\frac{s}{\delta}}} $ and, moreover,
$f(s) \leqslant   \frac{f(0)}{d} (d-s)_+$ for all $s \in [0,d]$.
\end{example}

\begin{proof}[Proof of \cref{lem-st-n}]
Let $g(s) \coloneqq \frac{f(s)}{f(0)}$. Then from \cref{nnn-1} we find
that
$$
g(s +\delta) \leqslant   A\, g^{\nu}(s) \quad \text{ for all } s \in \mathbb{R}^+, \ \delta > 0, \ g(0) =1,
$$
where $A = \epsilon f^{\nu -1}(0)$. Define
$$
s_{n+1} = s_n + \delta_n,\ s_0 = 0, \ \delta_n = f(0)A^n \quad \text{ for } n =0,1,2,\ldots,
$$
so
$$
g(s_{n+1}) \leqslant   A g^{\nu}(s_n).
$$
Then, by iteration, in view of $g(0) = 1$, we arrive at
$$
g(s_n) \leqslant   A^{1 + \nu + \nu^2+..+\nu^{n-1}}  , \  \
s_n = \mathop {\sum}_{k = 0}^{n-1} \delta_k  = f(0) \mathop {\sum}_{k = 0}^{n-1} A^k  ,
$$
whence
$$
0 \leqslant   g \Bigl( f(0) \mathop {\sum}_{k = 0}^{n } A^k \Bigr) \leqslant   g(s_n) \leqslant     A^{\frac{\nu^n -1}{\nu -1}} .
$$
Selecting $A < 1$ and letting $n \to +\infty$, we obtain that  $g(d)  = 0$ with $d = \frac{f(0)}{1-A}$, hence $f(d) = 0$.
\end{proof}

\begin{lemma}[Inhomogeneous Stampacchia's lemma]\label{lem:stampacchia2}
	Let $f:[0,R] \to \R$ be a non-negative non-increasing function such that
	\begin{equation}\label{stamglem:h21}
	f(\xi) \leqslant  \frac{c_0}{(\xi-\eta)^\alpha}(f(\eta)+\widetilde{S}\cdot (R-\eta)^{\frac{\alpha}{\beta-1}})^\beta,
	\end{equation}
	where $0 \leqslant  \eta < \xi \leqslant   R$, $\widetilde{S} \geqslant  0$, and $c_0,\alpha,\beta >0$, and $\beta >1$. In addition, let us assume that
	\begin{equation}\label{stamglem:h22} R^{\frac{\alpha}{\beta -1}} \geqslant  \left( 2^{\frac{\beta(\alpha + \beta -1)}{\beta -1 }} c_0\right)^{\frac{1}{\beta-1}}  (f(0) + \widetilde{S} \cdot R^{\frac{\alpha}{\beta -1}}).
	\end{equation}
	Then $$f(R) = 0.$$
\end{lemma}

\subsection{Fractional Leibnitz rule and tail estimates}

For fractional derivatives, the classical Leibnitz rule does not hold. However, by \cite[Theorem 1.2]{Dong} (see also \cite[Theorem 1.1]{VazquezHungSire}), we are able to estimate the remainder term as follows.

\begin{lemma}[Fractional Leibnitz rule]\label{lm:leibnitz}
Assume that $\beta \in (0,2)$. Let
$$
R_{\beta}(u,v) = (-\Delta)^{\beta}(u\,v) - u (-\Delta)^{\beta} v - v (-\Delta)^{\beta}u  .
$$
We have that
\begin{equation}\label{l-1}
\|  R_{\beta}(u,v) \|_{L^2(\Omega)}
\leqslant   C  \|   u  \|_{L^2(\Omega)} \| (-\Delta)^{\beta } v \|_{L^{\infty}(\Omega)}
\end{equation}
where $\Omega  \subset \mathbb{R}^d$ is a bounded domain,  $ \mathcal B[u] =0$ on $\partial \Omega$, and $\mathcal B[v]  = 0 $ on $\partial \Omega$. Here the boundary operator $\mathcal B[v]$ can be chosen as
\begin{align*}
\mathcal B[v] = v \quad \text{ or } \quad \mathcal B[v] =  \nabla v \cdot \textbf{n}.
\end{align*}
\end{lemma}

\begin{lemma}[Interpolation inequality]\label{lem-fr-3}
Let $R>0$ and $\Omega=B_R(0)$. For $S\in (0,R)$ we define $\Omega(S)\coloneqq \{x\in \R^d: S<|x|<R\}=\Omega\setminus \overline{B_S(0)}$.
For $\delta\in(0,R-S)$, let $\psi_{S,\delta}\in C^\infty(\overline\Omega)$ satisfying
\cref{e-9} and \cref{psiest}.
For any $\alpha \in (0, \frac{1+s}{2})$ there exists a constant $C>0$  such that  the following inequality holds:
\begin{equation}\label{main}
\|  \psi_{S,\delta} (-\Delta)^{\alpha} u \|_{L^2(\Omega )} \leqslant     \| (-\Delta)^{\frac{s+1}{2}} (u \psi_{S,\delta} ) \|_{L^2(\Omega )}^{\theta} \|   u  \psi_{S,\delta}  \|_{L^2(\Omega(S))}^{1-\theta} + \frac{C }{\delta^{2\alpha}}   \| u  \|_{L^2(\Omega )}
\end{equation}
for any $ u \in H^{s+1}_N(\Omega)$,
 where $ \theta = \frac{2\alpha}{s+1}$.
\end{lemma}

\begin{proof}[Proof of \cref{lem-fr-3}]

We apply \cref{dong-01} to
 $v = u \, \psi_{S,\delta}   \in  H^{s+1}_N(\Omega)$.
As
$$
 (-\Delta)^{\alpha}(u \, \psi_{S,\delta} ) = \psi_{S,\delta}  (-\Delta)^{\alpha} u  + u  (-\Delta)^{\alpha} \psi_{S,\delta}  + R_{\alpha}(u ,\psi_{S,\delta} )
$$
then, using \cref{dong-01}  and \cref{l-1}
with $\beta = \alpha$, we have
\begin{align*}
\| \psi_{S,\delta}  (-\Delta)^{\alpha}  u   \|_{L^2(\Omega)} &\leqslant  
 \| (-\Delta)^{\alpha} ( u\,\psi_{S,\delta}   ) \|_{L^2(\Omega)}
+ \| u (-\Delta)^{\alpha}  \psi_{S,\delta}    \|_{L^2(\Omega)} + \| R_{\alpha}(u ,\psi_{S,\delta} ) \|_{L^2(\Omega  )} \\ & \leqslant  
\| (-\Delta)^{\frac{s+1}{2}} ( u\,\psi_{S,\delta}   ) \|_{L^2(\Omega)}^{\theta} \|   u\,\psi_{S,\delta}  \|_{L^2(\Omega)}^{1-\theta} +
   \| u   (-\Delta)^{\alpha}  \psi_{S,\delta}    \|_{L^2(\Omega  )} + \frac{C}{\delta^{2\alpha}} \| u   \|_{L^2(\Omega  )} \\ &\leqslant  
\| (-\Delta)^{\frac{s+1}{2}} ( u\,\psi_{S,\delta}   ) \|_{L^2(\Omega )}^{\theta} \|   u \psi_{S,\delta}  \|_{L^2(\Omega )}^{1-\theta} +
\frac{C}{\delta^{2\alpha}} \| u   \|_{L^2(\Omega  )},
\end{align*}
which yields \cref{main}.
\end{proof}

Finally, we present a tail estimate for the fractional Laplacian.

\begin{lemma}[Tail estimate for the fractional Laplacian]\label{lem-fr}

Let $R>0$ and $\Omega=B_R(0)$. For $S\in (0,R)$ we define $\Omega(S)\coloneqq \{x\in \R^d: S<|x|<R\}=\Omega\setminus \overline{B_S(0)}$.
Let $\alpha \in (0,1)$, and
$\psi \in H_N^{2\alpha}( \Omega )$ be such that $\rm{supp}\,\psi \subseteq \Omega(S)$.
Then there exists a constant $C > 0$ depending on $\alpha$ such that for any $\delta \in (0, R-S)$,  the following estimate holds:
\begin{equation}\label{fr-1}
\| (-\Delta)^{\alpha} \psi  \|_{L^2(\Omega)} \leqslant   \| (-\Delta)^{\alpha} \psi  \|_{L^2(\Omega(S+\delta))} +  \frac{C}{\delta^{2\alpha }}\|\psi  \|_{L^2(\Omega)} .
\end{equation}
\end{lemma}

\begin{proof}[Proof of \cref{lem-fr}]
Note that
$$
\| (-\Delta)^{\alpha} \psi  \|^2_{L^2(\Omega)} = \| (-\Delta)^{\alpha} \psi  \|^2_{L^2(\Omega(S+\delta))} +
\| (-\Delta)^{\alpha} \psi  \|^2_{L^2(\Omega \setminus \Omega(S +\delta))}.
$$
So, we need to estimate $\| (-\Delta)^{\alpha} \psi  \|_{L^2(\Omega \setminus \Omega(S+\delta))}$.
By the semigroup representation formula for the spectral fractional Laplacian (see \cite{AbatangeloValdinoci}),
$$
(-\Delta)^{\alpha} \psi(x) = \frac{1}{\Gamma(-\alpha)} \int_0^{+\infty} { (e^{t\Delta}\psi(x) - \psi(x) ) \frac{1}{t^{1+\alpha}}}\, \d t,
$$
where $e^{t\Delta}\psi(x) = \int_{\Omega} { K(x,y,t) \psi(y)\,\d y}$
and $K(x,y,t) \simeq t^{-\frac{d}{2}} e^{-\frac{|x-y|^2}{4t}}$ is the kernel of the heat operator on $\Omega$.
Then,
$$
(-\Delta)^{\alpha} \psi(x) = \frac{1}{\Gamma(-\alpha)} \int_0^{+\infty} { t^{-1-\alpha} \int_{\Omega } { K(x,y,t) (\psi(y) -\psi(x))\,\d y}  \,\d t }
$$
which yields
\begin{align*}
|(-\Delta)^{\alpha} \psi(x) | &\leqslant   C \int_0^{+\infty} { t^{-\frac{d+2}{2}-\alpha} \int_{\Omega } { e^{-\frac{|x - y|^2}{4t}} |\psi(y) -\psi(x)| \,\d y} \, \d t } \\ &=
C \int_{\Omega } {  \frac{|\psi(y) -\psi(x)|}{|x-y|^{d +2\alpha}} \int_{0 }^{+\infty} { z^{\alpha + \frac{d-2}{2} }   e^{-z} \,\d z} \, \d y } \\ &\leqslant  
C  \| \psi \|_{L^2(\Omega)} \Bigl( \int_{\Omega(S)} {  \frac{dy}{|x-y|^{2d +4\alpha}} } \Bigr)^{\frac{1}{2}}
+ C |\psi(x)| \int_{\Omega } {  \frac{ \dd y }{|x-y|^{d +2\alpha}} }  ,
	\end{align*}
where $C  > 0$ depends on $\alpha$ and $d$. Hence,
\begin{align*}
\| (-\Delta)^{\alpha} \psi  \|^2_{L^2(\Omega \setminus \Omega(S +\delta))} &\leqslant  
C  \| \psi \|^2_{L^2(\Omega)} \int_{\Omega \setminus \Omega(S+\delta)} {  \Bigl( \int_{\Omega } {  \frac{\d y}{|x-y|^{2d +4\alpha}} } \Bigr) \,\d x}
\\ &\quad +   C \int_{\Omega \setminus \Omega(S+\delta)} { \psi^2(x)   \left( \int_{\Omega } {\frac{\d y}{|x-y|^{d +2\alpha}} } \right)^2 \,\d x} \leqslant    \frac{C}{\delta^{4\alpha }}  \| \psi \|^2_{L^2(\Omega)},
\end{align*}
which yields \cref{fr-1}.
\end{proof}

\vspace{0.5cm}
\section*{Acknowledgments}
NDN is a member of the Gruppo Nazionale per l’Analisi Matematica, la Probabilità e le loro Applicazioni (GNAMPA) of the Istituto Nazionale di Alta Matematica (INdAM).  He has been supported by the Swiss State Secretariat for Education, Research and Innovation (SERI) under contract number MB22.00034 through the project TENSE.

SL is a member of GNAMPA--INDAM.
SL acknowledges the support of the MIUR--PRIN grant no.~202244A7YL ``Gradient Flows and Non-Smooth Geometric Structures with Applications to Optimization and Machine Learning''.

AS is a member of GNAMPA--INDAM and acknowledges the support of the 2024 Project ``Local and nonlocal
variational models in complex materials''. Moreover AS acknowledges support from PRIN 2022 (Project no. 2022J4FYNJ), funded by MUR, Italy, and the European Union -- Next Generation EU, Mission 4 Component 1 CUP F53D23002760006.

RT was supported by NRFU project no.~2023.03/0074 ``Infinite-dimensional evolutionary equations with multivalued and stochastic dynamics'' and by a grant from the Simons Foundation (Award no.~00010584, Presidential Discretionary Ukraine Support Grants).

\vspace{0.5cm}
\textbf{Data Availability statement.} This manuscript has no associated data.

\textbf{Compliance with ethical standards.} The authors of this manuscript declare that they have no conflict of interest.


\bibliographystyle{abbrv}
\bibliography{FTFE-refs.bib}
\vfill
\end{document}